\documentclass[12pt, reqno]{amsart}
\usepackage[utf8]{inputenc}

\usepackage{amssymb}
\usepackage{graphics}
\usepackage{graphicx}
\usepackage{amscd}
\usepackage{color}
\usepackage{amsmath, amsthm, epsfig,  psfrag}
\usepackage{pinlabel}
\usepackage{hyperref}
\usepackage{subcaption}
\usepackage[margin=3cm, marginpar=2.5cm]{geometry}

\usepackage[all,pdftex,arc,curve,color,frame]{xy}

\newtheorem{theorem}{Theorem}[section]
\newtheorem{lemma}[theorem]{Lemma}
\newtheorem{prop}[theorem]{Proposition}
\newtheorem{cor}[theorem]{Corollary}

\theoremstyle{definition}
\newtheorem{definition}[theorem]{Definition}
\newtheorem{remark}[theorem]{Remark}
\newtheorem{example}[theorem]{Example}

\newcommand{\on}{\operatorname}
\renewcommand{\d}{\partial}
\renewcommand{\Im}{\on{im}}

\newcommand{\D}{\mathcal{D}}
\newcommand{\PP}{\mathcal{P}}
\newcommand{\inv}{^{-1}}
\newcommand{\eps}{\varepsilon}
\newcommand{\bb}{\mathbf{b}}

\newcommand{\CC}{\mathbb{C}}

\newcommand{\RR}{\mathbb{R}}
\newcommand{\ZZ}{\mathbb{Z}}
\newcommand{\QQ}{\mathbb{Q}}

\DeclareMathOperator{\tb}{tb}

\DeclareMathOperator{\Ker}{ker}

\title{Surface singularities and planar contact structures}

\author{Paolo Ghiggini}
\author{Marco Golla}
\address{CNRS, Laboratoire Jean Leray, Universit\'e de Nantes, Nantes, France}
\email{paolo.ghiggini@univ-nantes.fr}
\email{marco.golla@univ-nantes.fr}

\author{Olga Plamenevskaya}
\address{Department of Mathematics, Stony Brook University, Stony Brook, NY, U.S.A.}
\email{olga@math.stonybrook.edu}

\begin{document}

\begin{abstract}
We prove that if a contact 3-manifold admits an open book decomposition of genus 0, a certain intersection pattern cannot appear in the homology of any of its minimal symplectic fillings, and moreover, fillings cannot contain symplectic surfaces of positive genus.  
Applying these obstructions to canonical contact structures on links of normal surface singularities, we show that links of isolated singularities of surfaces in the complex 3-space are planar only in the case of $A_n$-singularities.
In general, we characterize completely planar links of normal surface singularities (in terms of their resolution graphs); these singularities are precisely rational singularities with reduced fundamental cycle (also known as minimal singularities).
We also establish non-planarity of tight contact structures on certain small Seifert fibered L-spaces and of contact structures arising from the Boothby--Wang construction applied to surfaces of positive genus. 
Additionally, we prove that every finitely presented group is the fundamental group of a Lefschetz fibration with planar fibers. 
\end{abstract}

 \maketitle

\section{Introduction and background} 

Since the groundbreaking work of Giroux~\cite{Gi}, open books have played a major role in 3-dimensional contact topology;
certain properties of open books are related to questions of tightness and fillability.
While a compatible open book decomposition is not unique, one can ask what the smallest possible genus of a page is.

In particular, contact manifolds that admit {\em planar} open book decompositions (i.e. with page of genus zero and possibly multiple boundary components) have a number of special properties.
For example, Etnyre showed that any symplectic filling for a planar contact structure has a negative definite intersection form~\cite{Etn};
it follows that any contact structure that arises as a perturbation of a taut foliation cannot be planar, since it admits fillings with arbitrary $b_2^+$~\cite{Eliashberg, Etnyre}.
This implies, by~\cite{RasmussenS, HRRR}, that if $Y$ is a graph manifold which is not an L-space, then $Y$ admits a non-planar contact structure.
(Recall that L-spaces, whose name derives from their Floer-homological similarity to lens spaces, are 3-manifolds with the simplest possible Heegaard Floer homology~\cite{OSlens}.)
By contrast, all contact structures on lens spaces are planar~\cite{Scho};
the same is known for some other L-spaces, although in general L-spaces can admit non-planar contact structures as well~\cite{LS3}. (Note that overtwisted contact structures are always planar by~\cite{Etn}.)

In this paper we develop new obstructions in terms of the topology of symplectic fillings: namely the presence of either a certain pattern in the intersection form or embedded symplectic surfaces with positive genus.
Using these obstructions, we rule out planarity for a number of interesting contact
structures. All contact manifolds in this paper are assumed closed and co-oriented. 

Before stating the general conditions, we interpret our obstructions for canonical contact structures on links of normal surface singularities.
Our first result is for isolated singularities of hypersurfaces in $\CC^3$, the second is for more general surfaces.
(The definitions involved in the second statement are more technical, and we defer them to Section~\ref{normal-links}.)

Consider a complex surface $\Sigma\subset \CC^N$ with an isolated critical point at the origin.
For a sufficiently small $\eps>0$, the intersection $Y= \Sigma \cap S^{2N-1}_{\eps}$ with the
sphere $S^{2N-1}_{\eps}= \{|z_1|^2+|z_2|^2+\dots +|z_N|^2 = \eps\}$ is a smooth 3-manifold called the {\em link of the singularity}.
The induced contact structure $\xi$ on $Y$ is the distribution of complex tangencies to $Y$, and is referred to as the {\em canonical} contact structure on the link.
The contact manifold $(Y, \xi)$ is independent of the choice of $\eps$, up to contactomorphism.

\begin{theorem}\label{sing}
Let $(Y, \xi)$ be the link of an isolated singularity of a complex surface in $\CC^3$ with its canonical contact structure. 
Then $\xi$ is planar if and only if the singularity is of type~$A_n$. 
\end{theorem}

\begin{theorem}\label{sing2}
Let $(Y,\xi)$ be the link of a normal surface singularity with its canonical contact structure.
Then $\xi$ is planar if and only if the singularity 
is  rational with reduced fundamental cycle.
\end{theorem}

Rational singularities with reduced fundamental cycle are also known as {\em minimal} singularities in the literature; they were first introduced by Koll\'ar~\cite{Ko}. We avoid the latter term because of possible confusion 
with its other common meaning (referring to the absence of symplectic $(-1)$-spheres, or exceptional divisors). These singularities can be defined in terms of their dual resolution graphs; namely, a rational singularity with reduced fundamental cycle    
has a good resolution whose graph is a tree of spheres with no bad vertices (see Definition~\ref{def-redfundc}).

Deformation theory of rational singularities with reduced fundamental cycle (and more generally, sandwiched singularities)  was studied in~\cite{djvs}. In a future work, we will discuss the theory of~\cite{djvs} from the symplectic topology viewpoint, using planar open books, and study the relation between smoothings  and
symplectic fillings of  canonical contact structures for this class of singularities.  

The ``if'' direction of Theorem~\ref{sing2} was essentially proven by Sch\"onenberger~\cite{Scho}, who established planarity of contact structures obtained by Legendrian surgery on Legendrian links associated to graphs with no bad vertices.
From a different perspective, contact structures on links of singularities whose graphs have no bad vertices were discussed by N\'emethi and Tosun~\cite{NT}, who showed that in this case the Milnor open books (associated to the Artin cycle $Z_{\rm min}$) are planar.
Milnor open books are open books that arise as follows: 
if $f$ is a holomorphic function on the complement of the singular point on the surface, the function $f/\|f\|$ defines a fibration on the link in the complement of the set $\{ f=0 \}$. 
(This set is the binding of the open book.) 
The Milnor open book supports the canonical contact structure in the link of the singularity by Caubel, N\'emethi, and Popescu-Pampu~\cite[Theorem 1.3]{CNPP}. A connection between the constructions of~\cite{NT} and~\cite{Scho} is given by Gay and Mark in~\cite{GayMark}, who give an explicit open book and factorization of its monodromy.

It is known that for general links of singularities, the smallest possible genus of compatible Milnor open books can be strictly greater than the minimal genus of arbitrary open books compatible with the canonical contact structure~\cite{Bu?}.
Together with~\cite{NT}, Theorem~\ref{sing2} shows that Milnor open books (in the sense of~\cite[Section~2.2]{NT}) minimize the genus in the planar case.

\begin{cor} Let $(Y,\xi)$ be the link of a normal surface singularity with its canonical contact structure. If $\xi$ is planar, then $\xi$ has a Milnor open book which 
is planar. 
\end{cor}

As a corollary of the proof of Theorem~\ref{sing2}, we obtain the following;
we say that a singularity is planar if the canonical contact structure on its link is planar.
\begin{cor}\label{deformation}
There can be no strong symplectic cobordism from a non-planar normal surface singularity to a planar one.
In particular, a planar normal surface singularity cannot be deformed to one whose link is not planar.
\end{cor}
This corollary goes in the direction of arguing that there can be no Weinstein cobordism from a non-planar contact structure to a planar one, or, more generally, that the support genus is non-decreasing under symplectic cobordisms.

For general planar contact manifolds, we show that a contact structure given by a plumbing graph with a bad vertex cannot be planar if the  vertices adjacent to the bad one have weight $-2$ or $-3$, as in conditions~\eqref{bad-form}.

\begin{theorem}\label{general}
A planar contact manifold {\em cannot} have a minimal symplectic filling $W$ with the following property:
for some $k>0$, there exist homology classes $B_1, \dots B_k, X \in H_2(W)$ such that 
\begin{equation}\label{bad-form}
\begin{array}{lll}
B_i \cdot X  = 1, & & \quad i =1, \dots, k \\
B_i \cdot B_j  =  0, & & \quad i \neq j \\
B_i \cdot B_i  \in  \{-2, -3\}, & & \quad i = 1, \dots, k\\
X \cdot X  > -k. & &
\end{array}
\end{equation}
In other words, the intersection graph of $W$ cannot have a configuration shown in Figure~\ref{bad}.
\end{theorem}

\begin{figure}[ht]
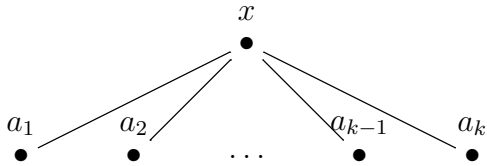

$\xygraph{
!{<0cm,0cm>;<1cm,0cm>:<0cm,1cm>::}
!~-{@{-}@[|(2.5)]}
!{(0,0) }*+{\bullet}="a1"
!{(1.5,0) }*+{\bullet}="a2"
!{(3,0) }*+{\dots}
!{(4.5,0) }*+{\bullet}="ak1"
!{(6,0) }*+{\bullet}="ak"
!{(3,1.5) }*+{\bullet}="c"
!{(0,0.4) }*+{a_1}
!{(1.5,0.4) }*+{a_2}
!{(4.5,0.4) }*+{a_{k-1}}
!{(6,0.4) }*+{a_k}
!{(3,1.9) }*+{x}
"c"-"a1"
"a2"-"c"
"ak1"-"c"
"ak"-"c"
}$
\caption{This intersection pattern cannot appear in the homology of a Stein filling of a planar contact structure if $-x < k$,   $a_i=-2$ or $a_i = -3$ for each $i$.
There can be more edges going out of each of the vertices labelled with $a_i$ and out of the central vertex.
All weights on additional vertices are also supposed to be negative.}
\label{bad}
\end{figure}

In particular, we have the following corollary for small Seifert fibered L-spaces.

\begin{cor}\label{seifert}
Tight contact structures on a Seifert fibered space $M(-2; r_1, r_2, r_3)$ are never planar if this manifold is an L-space and $r_1, r_2, r_3\geq \frac13$.
\end{cor}

Here we use the notation $M(e_0; r_1, r_2, r_3)$ for small Seifert fibered spaces;
$e_0 \in \ZZ$, $r_i \in (0,1)\cap \QQ$, and the space is given by the surgery diagram in Figure~\ref{seifert-space}.
Contact structures on these spaces  were extensively studied (see e.g.~\cite{Mat, Tos} and references therein). Many of them 
 turn out to be planar: every contact structure on $M(e_0; r_1, r_2, r_3)$ is planar if $e_0 \leq -3$~\cite{Scho}, and the same is true for $M(e_0; r_1, r_2, r_3)$ for $e_0\geq -1$ whenever this manifold is an L-space~\cite{LS3}.
Corollary~\ref{seifert} contrasts these planarity results.

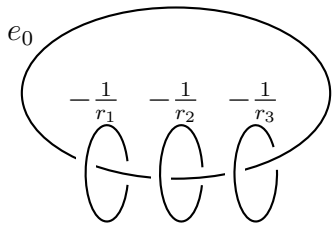
\begin{figure}[ht]
\labellist
\pinlabel $e_0$ at 0 140
\pinlabel $-\frac1{r_1}$ at 58 92
\pinlabel $-\frac1{r_2}$ at 115 92
\pinlabel $-\frac1{r_3}$ at 171 92
\endlabellist
\includegraphics[scale=0.5]{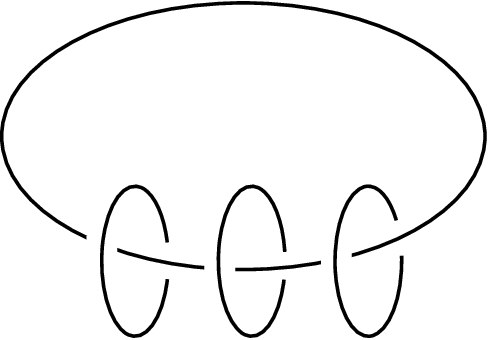}
\caption{The Seifert fibered space $M(e_0; r_1, r_2, r_3)$.}
\label{seifert-space}
\end{figure}

Our results stated above are special cases of a rather general examination of the homology of possible Stein fillings.
The major tool comes from Wendl's theorem, saying that any Stein filling of a planar contact manifold admits a Lefschetz fibration with the same planar fiber, and whose vanishing cycles can be obtained by a positive factorization of the monodromy of the planar open book into Dehn twists~\cite{We}.
(In particular, all curves along which Dehn twists are perfomed are homologically essential, which for planar pages is equivalent to being homotopically essential.)
Note that Wendl's theorem is extended to show that any minimal weak symplectic filling of a planar contact structure is deformation equivalent to a Stein filling admitting a Lefschetz fibration with the same properties~\cite{We2}.
Given a Lefschetz fibration, we can compute the homology and intersection form of the filling from the factorization of monodromy; 
together with the examination of all possible fillings, this leads to proofs of Theorems~\ref{sing} and~\ref{general}.  

In the proof of Theorem~\ref{sing2}, we consider the plumbing of {\em symplectic} surfaces and 
make use of the following result, which may be of independent interest. 

\begin{theorem}\label{positivegenus}
If a contact $3$-manifold $(Y,\xi)$ has a symplectic filling containing a symplectic surface of positive genus, it is not planar.
\end{theorem}
One could be tempted to prove this result by removing a tubular neighbourhood of the surface and invoking Etnyre's result on the nonexistence of symplectic semi-fillings of planar contact structure; see \cite[Theorem~1.2]{Etn}. However, a tubular neighbourhood of a symplectic surface has a concave boundary (and therefore, its complement has a convex boundary) if and only if its self-intersection is positive. This case is already ruled out by the negative definiteness of any filling of a planar contact structure (also from \cite[Theorem~1.2]{Etn}), while in the case of negative self-intersection the result is genuinely new.

As an immediate corollary of Theorem \ref{positivegenus}, we prove non-planarity of 
 Boothby--Wang contact structures for $g>0$. 
Given two integers $g\ge 0$ and $b>0$, we denote by 
$Y_{g,b}$ the total space of the circle bundle with Euler number $-b$ over a closed surface of genus $g$. Let $\xi_{g,b}$ be the Boothby--Wang contact structure on $Y_{g,b}$,
i.e. the contact structure induced on the (convex) boundary of the symplectic disk bundle over a surface of genus $g$, and Euler number $-b$;
see, for instance,~\cite{GayMark}. Since the 0-section of the disk bundle is symplectic,  we have 

\begin{cor}\label{c:prequantization}
Let $\xi_{g,b}$ be the Boothby--Wang contact structure on the circle bundle $Y_{g,b}$ with 
$b >0$. Then $\xi_{g,b}$ is planar if and only if $g = 0$.
\end{cor}

Let $\Sigma_{g,b}$ be the surface of genus $g$ with $b$ holes, and $\tau_{\partial}$ the boundary multi-twist, i.e. the product of right-handed Dehn twists along each boundary component. Then $\xi_{g,b}$ is supported by the open book $(\Sigma_{g,b}, \tau_{\partial})$, and thus Corollary \ref{c:prequantization} extends a partial result of Wand~\cite[Corollary 7.6]{Wand}.

It is useful to compare our obstructions to previous results. As noted above, Etnyre proved that any symplectic 
filling of a planar contact structure is negative definite, and that for a planar integral homology sphere, 
any symplectic filling must have a diagonal intersection form~\cite{Etn}.
This implies, for example, that the canonical contact structure on the Poincar\'e homology sphere 
(the link of the $E_8$-singularity) is not planar.
In fact, one can observe that Etnyre's proof yields a stronger statement: 
the intersection form of any symplectic filling of a planar \emph{rational} homology sphere embeds 
in a diagonal lattice of some (possibly higher) rank.
It follows that the canonical contact structures on the links of the $E_6$- and $E_7$-singularities cannot be planar, either.
On the other hand, Etnyre's result gives no information for the links of the $D_n$-singularities, 
as the corresponding intersection forms embed in the standard lattice. 

Another obstruction to planarity, in terms of Heegaard Floer homology, was developed by 
Ozsv\'ath--Stipsicz--Szab\'o~\cite{OSS}.
This obstruction is also trivial for the links of the $D_n$-singularities.
More generally, the Heegaard Floer obstruction is always trivial for L-spaces.
By contrast, our obstruction often gives non-trivial information in the case of L-spaces, see Corollary~\ref{seifert} above.

Using factorizations of mapping classes, Wand gave another obstruction to planarity~\cite{Wand}. Wand's results are closer in
spirit to ours, as he also uses Wendl's theorem and examines topology of fillings, however both the specific approach and 
the obstruction Wand obtains are different from ours. In particular, Wand shows that the sum of the Euler characteristic and 
signature is the same for all Stein fillings of a planar contact manifold.
Then, if one is able to find two weak symplectic fillings $W_1$, $W_2$ for $(Y, \xi)$ such that $\chi(W_1)+\sigma(W_1) \neq  \chi(W_2)+\sigma(W_2)$, it follows that $(Y, \xi)$ cannot be planar.
(Wand also examines how certain relators in the mapping class group affect $\chi+\sigma$.)
However, this obstruction fails to address the case when there is a unique filling; 
for example, it is known that the filling is unique for the links of the $D_n$-singularities~\cite{OhtaOno}, so Wand's approach gives no obstruction. 
Wand's obstruction is also trivial when the underlying contact 3-manifold is a rational homology sphere, and all its fillings are negative definite (this is true, in particular, for all L-spaces);
indeed, for a negative definite Stein filling $W$ of a rational homology sphere we always have $\chi(W)+\sigma(W)=1$ since $b_3(W)=b_1(W)=0$.
We are also able to answer a question of Wand in our Corollary~\ref{c:prequantization}, proving non-planarity for a family of contact structures that cannot be handled by Wand's means (see~\cite[Corollary 7.6]{Wand} and subsequent discussion).

As a byproduct of our intersection form calculation, we also get the following corollaries.
These were first proven by Oba~\cite[Lemma 3.1, Lemma 3.2]{Oba2} using the Heegaard Floer obstruction from~\cite{OSS}. It is 
instructive to obtain these results more directly, from the basic topology of fillings.  

\begin{cor}\label{blow-down}
Let $(W,\omega)$ be a weak symplectic filling of a planar contact manifold $(Y, \xi)$.
If $B \in H_2(W)$ is a class of square $-1$, $B$ is represented by an embedded symplectic sphere that can be blown down. 
\end{cor}

\begin{cor}\label{IHS}
Let $Y$ be an integral homology sphere, equipped with a planar contact structure 
$\xi$. Then any minimal weak symplectic filling of $(Y,\xi)$ is an integral homology ball.
\end{cor}

Non-trivial examples of fillings as in Corollary~\ref{IHS} do exist; a number
of examples were constructed by Oba~\cite{Oba}. More generally, we show that
one can construct Stein fillings with prescribed fundamental
groups.

\begin{prop}\label{p:ballgroups1}
Every finitely presented group is the fundamental group of the total space of a Lefschetz fibration over the disk with planar fibers.
\end{prop}

A more precise version of this statement, yielding also examples for Corollary~\ref{IHS}, is given in Proposition~\ref{p:ballgroups}. Note that Proposition~\ref{p:ballgroups1} is similar to a theorem of Amor\'os, Bogomolov, Katzarkov, and Pantev~\cite{ABKP} and to Gompf's theorem~\cite{Go}:
Gompf showed that any finitely presented group is the fundamental group of a closed symplectic 4-manifold, and in~\cite{ABKP},  a closed symplectic 4-manifold with prescribed fundamental group is constructed as a symplectic Lefschetz fibration over a closed surface. 
Unlike~\cite{Go, ABKP}, where no bounds are given for the genus of the fiber, we work with manifolds with boundary but restrict to Lefschetz fibrations with planar fibers.

\subsubsection*{Organization:}
In Section~\ref{compute}, we explain how to compute 
the intersection form and first Chern class of the filling 
constructed from a positive factorization of the monodromy of a planar open book, 
and prove Corollaries~\ref{blow-down} and~\ref{IHS}, and Theorem~\ref{positivegenus}.
In Section~\ref{links}, we prove Theorem~\ref{sing} (after considering the key example of $D_4$).
In Section~\ref{therest}, we prove Theorem~\ref{general}
and Corollary~\ref{seifert}.
In Section~\ref{normal-links}, we prove Theorem~\ref{sing2}.
Finally, in Section~\ref{pi1} we discuss fundamental groups and prove Proposition~\ref{p:ballgroups1}.
 
\subsubsection*{Acknowledgements:}
The authors would like to thank the Isaac Newton Institute for Mathematical Sciences for support and hospitality during the programme ``Homology theories in low dimensional topology'' when this work was undertaken.
This work was supported by EPSRC grant no EP/K032208/1.
OP is grateful to Andr\'as Stipsicz for a helpful conversation that led to Corollary~\ref{IHS};
MG thanks Cagri Karakurt, Andr\'as N\'emethi, Burak Ozbagci, Matteo Ruggiero, and Laura Starkston for interesting discussions and for giving many relevant references;
we also thank J\'ozsef Bodn\'ar for some motivating discussions.
PG and MG are partially supported by the ANR project QUANTACT (ANR16-CE40-0017).
MG was also supported by the Alice and Knut Wallenberg foundation and by the European Research Council (ERC) under the European Unions Horizon 2020 research and innovation programme (grant agreement No 674978).
OP was partially supported by NSF grant DMS-1510091.

\section{Computing topological invariants of a planar Lefschetz fibration}\label{compute}

In this section, we explain how to compute the intersection form and the first Chern class for a Lefschetz fibration 
over a disk with planar fibers. The second homology classes of a Lefschetz fibration are given by certain linear combinations 
of the vanishing cycles, and both the intersection form and the evaluation of the first Chern class can be found 
directly in terms of the vanishing cycles.  This is a consequence of fairly straightforward topological considerations, and we think 
that these facts, especially Proposition~\ref{iform}, should be known to experts (see Remark~\ref{Auroux-rmk} below), 
but complete statements and the proofs seem to be absent from the literature.  

Let $P$ be the page of a planar open book decomposition of $Y$; $P$ is the disk $D$ 
with a few holes. In this paper, the 3-manifold $Y$ is always assumed to be oriented, and the open book decomposition we consider is compatible with a co-oriented
contact structure. An orientation of contact planes, together with an orientation of $Y$, determine the orientation of the binding of the open book and the 
orientation of the page $P$. 
Let us assume that the monodromy $\phi$ of the open book is the product of
positive Dehn twists about homologically non-trivial simple closed curves
$\alpha_1,\dots,\alpha_m$ in $P$ for some $m$.
Each curve $\alpha_i$ divides $P$ into two components, and we \emph{orient} it as the boundary of the region $A_i$ disjoint from $\partial D$.
With this orientation, $\alpha_i$ defines a class in $H_1(P)$, that we denote with $[\alpha_i]_P$.
For convenience, we will also assume that the $\alpha_i$ are smoothly embedded and that they intersect transversely.
This implies that the union of the $\alpha_i$ disconnects $P$ into finitely many connected components.
Unless otherwise stated, homology is taken with coefficients in $\ZZ$.

Let $W$ be the total space of a Lefschetz fibration over a disk $\D$, with planar fiber $P$. We assume that $W$, $P$, and $\D$ are compatibly oriented.
(If $\d W$ is equipped with a contact structure and an open book with fiber $P$ as above, the orientations of $W$ and $\D$ are uniquely determined.)
If $\D' \subset \D$ is a small disk that contains no critical points, then $W$ is obtained from $P \times \D'$ by attaching 2-handles to copies of the vanishing cycles contained in the vertical boundary $P\times \d\D'$ so that distinct handles are attached along knots contained in distinct fibers.

We first describe $H_2(W)$ and give a convenient way to visualize homology classes.
We use the exact sequence of the pair $(W, P \times \D')$; since $P \times \D'$ retracts onto $P$, we can replace the former with the latter:
\[
0 \longrightarrow H_2(W) \stackrel{j_*}\longrightarrow H_2(W, P) \stackrel{\d_*}\longrightarrow  H_1(P) \longrightarrow H_1(W) \longrightarrow H_1(W,P) = 0.
\]
The group $H_2(W, P)$ is freely generated by the cores of the attached
2-handles;
we can identify these generators with the vanishing cycles.
Next, $H_2(W)$ is isomorphic to $\Im j_*$, which in turn equals $\Ker \d_*$.
So $H_2(W)$ can be identified with null-homologous linear combinations of vanishing cycles (thought of as 1-chains in $P$). 

Further, in $H_1(P)$ a linear combination $b_1[\alpha_1]_P + b_2[\alpha_2]_P+\dots +b_m[\alpha_m]_P$ is null-homologous if and only if the total winding number at each hole of $P$ is zero.
Notice that the curves correspond to distinct vanishing cycles, but their homology classes may coincide.
In our setting, each $\alpha_i$ is a vanishing cycle, so it is a simple closed curve on the planar 
surface $P$; then, with the chosen orientation convention, each $\alpha_i$ has winding number $0$ or $1$ at each hole.

We denote by $[\alpha_\bb]_W$ the homology class in $H_2(W)$ corresponding to the linear combination $\alpha_\bb = \sum b_i\alpha_i$;
we stress that whenever we write $[\alpha_\bb]_W$ it will be implicit in the notation that $[\alpha_\bb]_P = \sum b_i[\alpha_i]_P = 0 \in H_1(P)$.
Consider the linear combination of $2$-chains $A_\bb = \sum b_i A_i$, where as above $A_i$ is the oriented region in $D$ with $\d A_i= \alpha_i$. 
While it is possible to compute the self-intersection already at this point, using transversality for singular chains,
we find it more satisfactory to represent homology classes by oriented embedded surfaces as follows.

Since $[\alpha_\bb]_P = 0 \in H_1(P)$, the 2-chain $A_\bb$ has multiplicity 0 near each boundary component of $P$. 
Let $\alpha_\bb$ be the boundary of $A_\bb$; we construct a surface corresponding to $\alpha_\bb$ as follows.
Let $\D''=\frac12 \D' \subset \D'$ be a smaller disk, and identify its boundary with $S^1\subset\CC$. 
Let $|\bb| = \sum |b_i|$, and consider $|\bb|$ fibers $P_1,\dots, P_{|\bb|}$ of $P\times \D'$
over the points $\eta_j = \exp(2\pi ji/|\bb|)$ in  $\partial \D''$. 
Rewrite the sum $\sum b_i A_i$ as $\varepsilon_1A_{i_1} + \dots + \varepsilon_{|\bb|}A_{i_{|\bb|}}$, where each $\varepsilon_i$ is $\pm1$.
Using the same notation, we will consider $A_{i_j}$ as a copy of the corresponding region in the fiber $P_j$ over $\eta_j$.   

Look at a hole $h$ of $P$.
We ignore all indices $i$ such that $\alpha_{i}$ has winding number 0 around $h$, since the corresponding 2-chain $A_i$ is disjoint from $h$.
Since the winding number of $\alpha_\bb$ around $h$ is 0, all other indices, considered with their multiplicity, can be paired up; 
more precisely, we can rewrite $A_\bb = A_{j_1}-A_{k_1} + \dots + A_{j_n} - A_{k_n} + A'$, where $A'$ is a 2-chain disjoint from the hole $h$. 
Note that as before, all  2-chains in the expression for $A_\bb$ are the regions in the disk $D$, oriented compatibly with $P$ if they come with a $+$ sign, 
and oppositely otherwise.   

Using a standard innermost argument and connecting the paired-up 2-chains by oriented tubes, we can actually tube 
away all intersections of $A_\bb$ with 
$h\times \partial\D''$ by adding cylinders that are parallel to $\partial h \times \partial\D''$ in 
$P \times \partial\D''$.  
The result is an oriented embedded surface in $P \times \D'$ whose boundary consists of a number of vanishing cycles.
In $W$, vanishing cycles are null-homologous, so they can be capped off with a Lefschetz thimble to make an oriented embedded closed surface in $W$ representing the homology class $[\alpha_\bb]_W$.
See Figure~\ref{surface}.

\begin{figure}[ht]
\centering
    \begin{subfigure}[t]{0.4\textwidth}
    \labellist
    \pinlabel $A_2$ at 120 120
    \pinlabel $A_3$ at 450 120
    \pinlabel $A_1$ at 450 450
    \endlabellist
        \includegraphics[width=\textwidth]{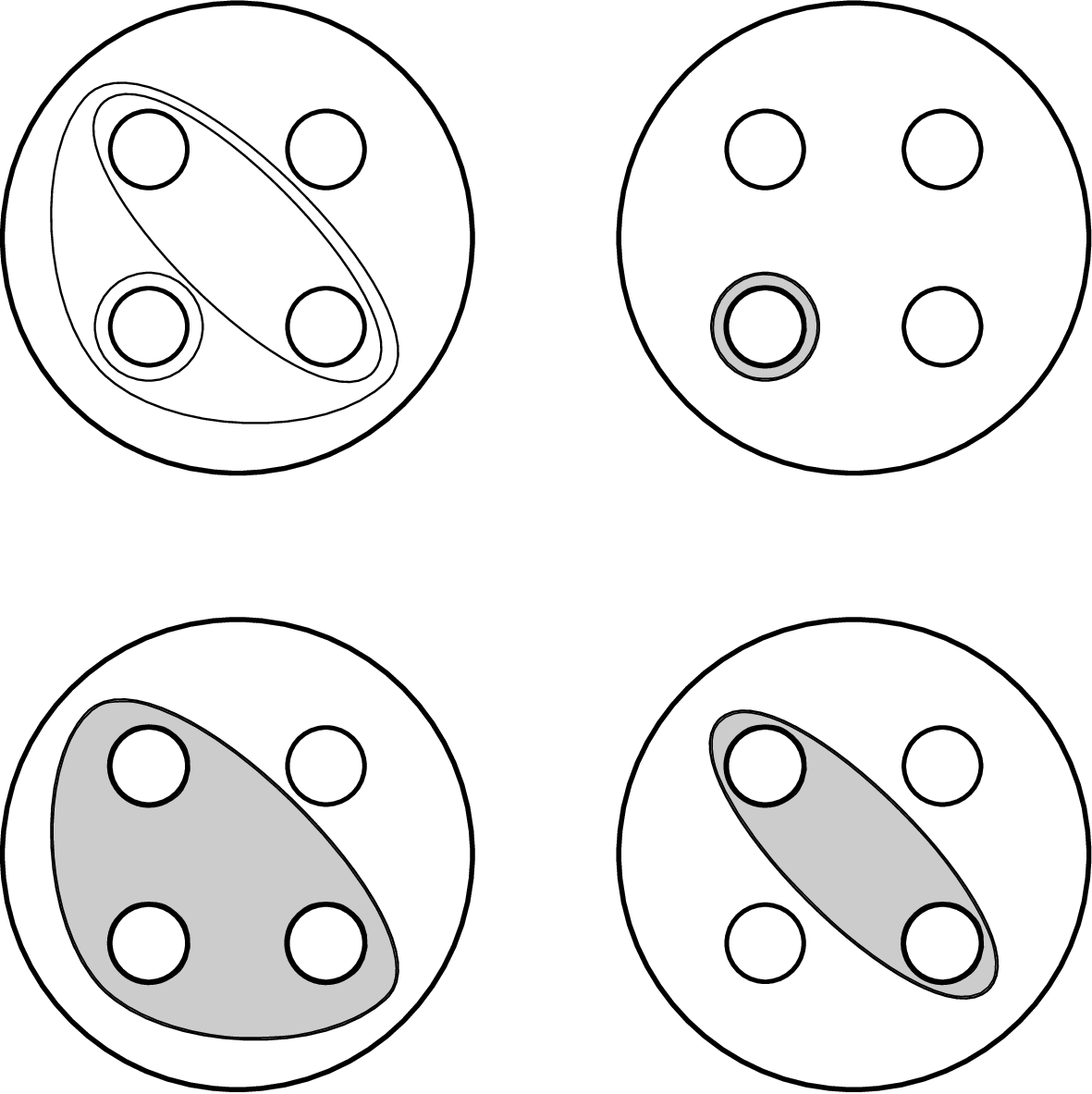}
        \caption{The three curves $\alpha_1, \alpha_2, \alpha_3$ and the corresponding regions $A_1, A_2, A_3$.}
        \label{f:H1LF1}
    \end{subfigure}
\hspace{0.1\textwidth}
    \begin{subfigure}[t]{0.4\textwidth}
    \labellist
    \pinlabel $\eta_2$ at 160 60
    \pinlabel $\eta_1$ at 270 110
    \pinlabel $\eta_3$ at 415 90
    \endlabellist
        \includegraphics[width=\textwidth]{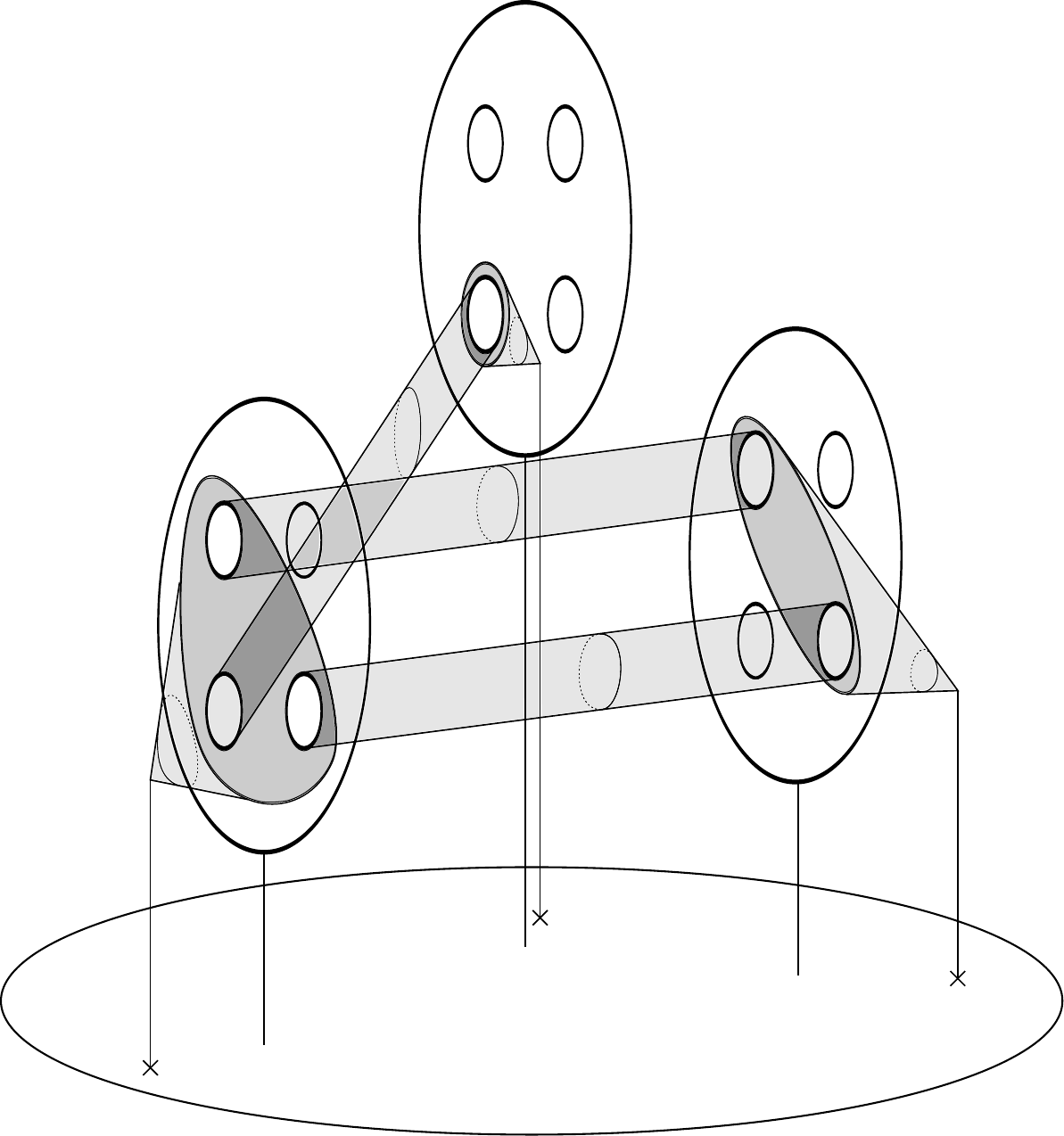}
        \caption{The surface representing the homology class: the three $\times$ denote the critical point of the projection associated to $\alpha_1, \alpha_2, \alpha_3$, and the corresponding cones are the three Lefschetz thimbles.}
        \label{f:H1LF2}
    \end{subfigure}
  \caption{Constructing an embedded surface representing the homology class $[\alpha_2-\alpha_1-\alpha_3]_W$.}
  \label{surface}
\end{figure}

Now we can determine the intersection form of $W$, by computing the intersection of two classes.
\begin{prop}\label{iform}
Consider two homology classes $B = [b_1\alpha_1 + \dots +b_m\alpha_m]_W$, $B'=[b'_1\alpha_1 + \dots +b'_m\alpha_m]_W$ in $H_2(W)$. 
Then 
\[
B \cdot B' = -(b_1b_1'+\dots+b_mb'_m).
\]
In particular, $B\cdot B = -(b_1^2+\dots+b_m^2)$.
\end{prop}

\begin{proof}
We construct representatives of $B$ and $B'$ as above, starting with fibrations over disjoint small disks $\D'_1$ and  $\D'_2$ away from the critical points.
The parts of the surfaces contained in $P \times \D'_1$ resp. $P \times \D'_2$ are then disjoint, but 
intersections may appear after we cap off the vanishing cycles on the boundary of these surfaces.
This is schematically depicted in Figure~\ref{f:spiders}.
Intersections now come in two sorts: (i) the self-intersection of the cap (thimble) corresponding to the vanishing cycle $\alpha_i$, and (ii) the
intersection of the caps corresponding to the distinct vanishing cycles $\alpha_i$ and $\alpha_j$.
In case (i), this self-intersection equals, by construction, the framing of $\alpha_i$ along which the corresponding $2$-handle is attached, relative to the page $P$. 
A standard result in topology of Lefschetz fibrations (see e.g.~\cite[Section 8.2]{Gompf-Stipsicz}) says that for each critical point, the corresponding 2-handle is attached along the vanishing cycle with the framing $-1$ relative to the page framing.
This implies that the self-intersection of each thimble is $-1$.

In case (ii), the intersection of caps is given by the intersection of the curves $\alpha_i$ and $\alpha_j$ on the page $P$. Indeed, the cap for $\alpha_i$ is a thimble that connects a copy of the curve $\alpha_i$ in the fiber over some point $\eta_{s}$ with $i_s = i$ to the corresponding critical point in a singular fiber, see 
Figure~\ref{surface}.
The projection of this thimble to the base disk is a path from $\eta_{s}$ to the critical value in the disk. Similarly, the projection of the cap for $\alpha_j$ 
is a path from a point $\eta_{t}$ with $i_t = j$ to the critical value corresponding to the vanishing cycle $\alpha_j$, see Figure~\ref{f:spiders}. 
The two caps will be disjoint if these two paths are 
disjoint. If the paths intersect at a point $p \in \D$, the intersection of the two thimbles with the fiber $P_p$ over $p$ is given by curves isotopic to 
$\alpha_i$ resp. $\alpha_j$.  The intersection of caps is then given by the sum of contributions of all such points $p$, and each intersection point of the paths 
contributes  $[\alpha_i]_P \cdot [\alpha_j]_P$ to the sum. 
Since $P$ is planar, we see that $[\alpha_i]_P \cdot [\alpha_j]_P =0$ for the simple closed curves $\alpha_i$ and $\alpha_j$; thus, the total contribution is 0 in case (ii).
\end{proof}

\begin{figure}[h]
\includegraphics[scale=0.7]{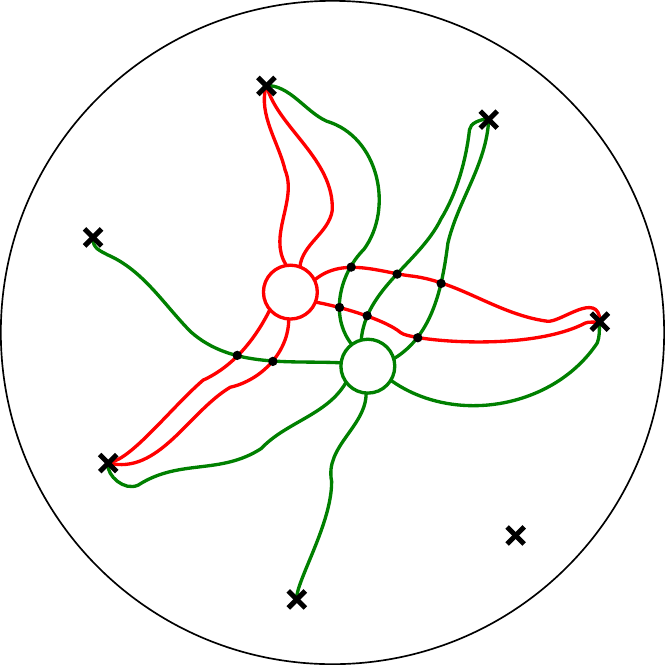}
\caption{The red and green figures are meant to represent two embedded surfaces constructed from two linear combinations of cycles as above. The tubes lie over the green and red circles.
Intersections happen over the regular fibers (dotted) and over the vanishing cycles (crossed).}\label{f:spiders} 
\end{figure}

\begin{remark}\label{Auroux-rmk}
In case (ii) we use planarity in an essential way; 
the formula would have additional terms for a higher-genus page.
A version of Proposition~\ref{iform} holds in the case of a higher-genus fiber, with a similar proof, but there are extra terms given by the intersections $[\alpha_i]_P \cdot [\alpha_j]_P$ that can be non-trivial in general.
This fact is mentioned, without proof, in the course of the proof of~\cite[Lemma~16]{Aur}.
We focused on the planar case since it is sufficient for our purposes, the statement is simpler, and the surface representatives are easier to visualize.
\end{remark}

\begin{cor}\label{lessthan1}
Let $W$ be a Stein filling of a planar contact $3$-manifold, and $B \in H_2(W)$ a non-zero homology class. 
Then $B \cdot B \leq -2$.
\end{cor}
  
\begin{proof}
By Wendl's theorem~\cite{We}, $W$ corresponds to a factorization of the monodromy into positive Dehn twists along curves $\alpha_1,\dots,\alpha_m$.
We know by~\cite{Etn} that $B \cdot B < 0$.
The class $B$ is $[b_1\alpha_1 + \dots +b_m\alpha_m]_W$, where not all of the coefficients $b_i$ vanish.
Suppose that  $B\cdot B = -(b_1^2+\dots+b_m^2) = -1$, this implies that all $b_i$ vanish except for one, say $b_1 = \pm1$.
But this contradicts the fact that all vanishing cycles are homologically essential in $P$.
\end{proof}
         
We immediately get Corollaries~\ref{blow-down} and~\ref{IHS}. 
         
\begin{proof}[Proof of Corollary~\ref{blow-down}]
Let $W$ be a weak symplectic filling of a planar contact manifold $(Y, \xi)$.
Suppose that $B\cdot B=-1$ for a class $B \in H_2(W)$.
By~\cite{We2}, if $W$ were minimal, $W$ would be deformation equivalent to a Stein filling (given by a Lefschetz fibration with planar fibers), and so the previous corollary would give a contradiction.
Suppose now that $E$ is the homology class of a symplectic $(-1)$-sphere in $W$;
if $E = \pm B$ the proof is complete.
We claim now that, if $E\neq \pm B$, then $E\cdot B = 0$.
To this end, let $B\cdot E = x$, and look at the subspace of $H_2(W)$ generated by $E$ and $B$;
the intersection form of $W$, restricted to this subspace, is $\big({\tiny\begin{array}{cc}-1&x \\ x&-1\end{array}}\big)$; by~\cite{Etn}, this matrix has to be negative definite, and this can happen if and only if $x=0$.
It follows that $W$ can be blown down along a sphere in $E$, and that, in the blowdown, $B\cdot B = -1$.
By induction, we can blow down to a minimal weak filling $W_0$; since this can be deformed to a Stein filling, Corollary~\ref{lessthan1} now gives a contradiction.
\end{proof}
         
\begin{proof}[Proof of Corollary~\ref{IHS}]
Let $(W,\omega)$ be a minimal weak symplectic filling of $(Y,\xi)$.
By~\cite{We2}, $(W,\omega)$ can be deformed to a Stein filling; hence it has a handle decomposition with no 3-handles, so that $H_3(W) = 0$ and $H^3(W) \cong H_1(W,Y) = 0$.
Since $Y$ is an integral homology sphere, by Poincar\'e--Lefschetz duality the intersection form $Q$ of $W$ is unimodular.
From the long exact sequence of the pair $(W,Y)$, it also follows that $H_1(W) = 0$.
Again, by Poincar\'e--Lefschetz duality and the universal coefficient theorem, $H_2(W)$ is torsion-free.

By results of Etnyre~\cite{Etn}, since $Y$ is an integral homology sphere and $\xi$ has a planar open book decomposition, then the intersection form $Q$ of $W$ embeds in the negative definite diagonal lattice $\ZZ^N$ for some $N$.
Since $Q$ is unimodular, $Q$ is in fact a direct summand of $\ZZ^N$, and in particular it is itself diagonalizable.
Therefore, unless the filling is a rational homology ball, $H_2(W)$ must have a class with self-intersection $-1$, but this is not possible by the previous corollary.
\end{proof}

There are many examples of planar, fillable integral homology spheres that are
not contactomorphic to the standard tight $S^3$; we discuss these in
Section~\ref{pi1}.

We now turn to the calculation of the first Chern class $c_1(J)$ for a compatible almost-complex structure on the Lefschetz fibration. Although planarity is crucial in the next proposition, much of the proof follows the lines of the well-known calculation of $c_1$ for Stein domains corresponding to Legendrian surgeries~\cite[Proposition 2.3]{Gompf-handle}.   

\begin{prop}\label{c1}
Let $(Y,\xi)$ be the contact structure associated to the planar open book $(P, \phi)$. 
Let $(W,\omega)$ be the symplectic filling of $(Y,\xi)$ associated to the factorization of $\phi$ into 
positive Dehn twists along the curves $\alpha_1,\dots,\alpha_m$, oriented coherently with the outer boundary of $P\subset D^2$.
If $J$ is an almost-complex structure compatible with $\omega$, then
\[
\langle c_1(J), [b_1\alpha_1 + \dots + b_m\alpha_m]_W\rangle = b_1 + \dots + b_m.
\]
\end{prop}

Note that something similar follows by work of Gay--Stipsicz~\cite[Corollary 2.3]{GayStipsicz}; they observe that, up to deformation, $(W,\omega)$ embeds in the complement of a line in a blowup $X$ of $\mathbb{CP}^2$.
Therefore $H_2(W)$ embeds in the lattice generated by the homology classes of (some of) the exceptional divisors of $X$; the first Chern class evaluates as 1 on each of these divisors, thus recovering an analogue of Proposition~\ref{c1}.
However, there are examples of Stein 4-manifolds that admit such an embedding, but are nevertheless \emph{not} planar;
for instance, the following 4-manifold is realised as a subdomain in blowup of $\CC^2$, as the corresponding embedding shows, but the planarity of its boundary is excluded by Theorem~\ref{sing2}.
\[
\xygraph{
!{<0cm,0cm>;<1cm,0cm>:<0cm,1cm>::}
!~-{@{-}@[|(2.5)]}
!{(0,0) }*+{\bullet}="a1"
!{(1.5,0) }*+{\bullet}="a2"
!{(-1.06,1.06) }*+{\bullet}="c"
!{(-1.06,-1.06) }*+{\bullet}="d"
!{(0,0.4) }*+{-2}
!{(1.5,0.4) }*+{-2}
!{(-1.06,1.46) }*+{-3}
!{(-1.06,-1.46) }*+{-2}
"c"-"a1"
"a2"-"a1"
"d"-"a1"
} \qquad = \qquad 
\xygraph{
!{<0cm,0cm>;<1cm,0cm>:<0cm,1cm>::}
!~-{@{-}@[|(2.5)]}
!{(0,0) }*+{\bullet}="a1"
!{(1.5,0) }*+{\bullet}="a2"
!{(-1.06,1.06) }*+{\bullet}="c"
!{(-1.06,-1.06) }*+{\bullet}="d"
!{(0.2,0.4) }*+{\phantom{}_{e_3 - e_4}}
!{(1.5,0.4) }*+{\phantom{}_{e_2 - e_3}}
!{(-1.06,1.46) }*+{\phantom{}_{e_1-e_2-e_3}}
!{(-1.06,-1.40) }*+{\phantom{}_{e_4-e_5}}
"c"-"a1"
"a2"-"a1"
"d"-"a1"
}
\]

\begin{proof}
As before, the space $W$ is obtained from $P \times \D'$ by attaching 2-handles. 
The complex bundle $(TW, J)$ is trivial over $P \times \D'$, and $c_1(J)$ measures the obstruction to extending a trivialization over the 2-handles. 
We will argue that for each 2-handle, this obstruction is the same in the appropriate sense.
We can assume that the $2$-handles are attached to fibers of $P\times \D'$ over points in a small arc in $\partial \D'$.
Fix an embedding $P\subset \CC$ and trivialize the complex bundle $T(P \times \D')= TP \times T\D'$ over the chosen fibers by a framing $(u, v)$, where $u$ is a constant vector field in $P \subset \CC$ and $v$ is an inward normal to $\partial \D'$ in $\D' \subset \CC$.
This trivialization extends to a complex trivialization  of  $T(P \times \D')$ over the entire product $P \times \D'$.
Each 2-handle $H_k$ can be identified with a fixed copy of $D^2 \times D^2 \subset i\RR^2 \times \RR^2$, and we can pick a complex trivialization of its tangent bundle that restricts to the circle $S^1\times 0 \subset H_k$  as the framing $(\tau, \nu)$,  where $\tau$ is the tangent and $\nu$ the outward normal vector fields to $S^1 = \partial D^2 \subset i \RR^2$.
(Indeed, the framing $(\tau ,\nu)$ differs from the restriction of the product framing to $S^1$ by an element of $\pi_1(SU(2))$,  and so $(\tau ,\nu)$ can be extended over the entire handle since $SU(2)$ is simply connected, see~\cite[Proposition 2.3]{Gompf-handle}).
When we attach the handle by identifying $S^1 \times 0 \subset H_k$ with the vanishing cycle $\alpha_k$, $\nu$ is identified with $v$, and $\tau$ is 
identified with the tangent vector field to $\alpha_k$.
Therefore, $\nu$ and $v$ together span a trivial complex line bundle, and $c_1(J)$ equals the first Chern class of the complex line bundle defined by $\tau$ and $\mu$.
To evaluate the latter on the core of the handle $H_k$ (as a relative Chern class), we must look at the rotation number of $\mu$ relative to $\tau$ along the vanishing cycle $\alpha_k$ in the page $P$.
Since $P$ is planar and $\alpha_k$  is a simple closed curve in $P$, it is clear that this rotation number equals $r=\pm 1$.
(The sign depends on the orientation conventions).
Note also that the value of $r$ is the same for all handles, since the tangent bundles over different pages are identified by our choice of trivialization, and different vanishing cycles in the same page $P \subset \CC$ are identified via an isotopy in $\CC$.  
It follows that
\begin{equation}\label{e:c1partial}
\langle c_1(J), [b_1\alpha_1 + \dots + b_m\alpha_m]_W\rangle = r(b_1 + \dots + b_m),
\end{equation}
where $r=\pm 1$.

To pin down the sign, we consider the lens space $L(3,1)$.
The canonical contact structure $\xi_0$ on $L(3,1)$ is the Boothby--Wang structure associated to the disk bundle over $S^2$ with Euler number $-3$.
As mentioned in the introduction, $\xi_0$ is supported by the open book on the 2-holed disk $P$, where the monodromy $\phi$ is the multi-twist along the boundary;
more precisely, Gay and Mark~\cite{GayMark} show that this factorisation corresponds (up to deformation equivalence) to the symplectic disk bundle filling $(W,\omega)$ of $\xi_0$.
As above, call $J$ an almost-complex structure compatible with $\omega$.

Note that the 0-section of the disk bundle is a symplectic sphere $S$, whose homology class generates $H_2(W)$;
in particular, $S$ satisfies the adjunction formula, and the symplectic form integrates positively over it.

The class $[\alpha_1-\alpha_2-\alpha_3]_W$, where $\alpha_1$ is parallel to the outer boundary, generates $H_2(W)$.
Since the corresponding 2-chain is made by a part of the page (with positive orientation) and 
three vanishing cycles (where $\omega$ vanishes), the symplectic form integrates positively over this linear combination.

In particular, $S = [\alpha_1-\alpha_2-\alpha_3]_W$, with its symplectic orientation.
Applying the adjunction formula and~\eqref{e:c1partial}
\[
-r = \langle c_1(J), [\alpha_1-\alpha_2-\alpha_3]_W \rangle = \langle c_1(J), [S] \rangle = S\cdot S + \chi(S) = -1,
\]
hence $r=1$, as claimed.
\end{proof} 

We now use Propositions~\ref{iform} and~\ref{c1} to prove Theorem~\ref{positivegenus}.
\begin{proof}[Proof of Theorem~\ref{positivegenus}]
Suppose that $(W,\omega')$ is a symplectic filling that contains a symplectic surface of genus $g > 0$,
and call $A$ its homology class.
Let $J'$ be an almost-complex structure compatible with $\omega'$.
Without loss of generality, we may assume that $(W, \omega')$ is minimal; then
Wendl's theorem guarantees that there is a deformation from $(\omega',J')$
to $(\omega,J)$, such that $(W,\omega)$ is supported by a planar Lefschetz fibration, corresponding to a factorization of a planar monodromy $\phi$ of $\xi$ into positive Dehn twists along $\alpha_1,\dots,\alpha_m\subset P$.

Note that $\langle c_1(J),A\rangle = \langle c_1(J'),A\rangle$, since evaluation of $c_1(J)$ can only take discrete values.
Suppose $A = [b_1\alpha_1+\dots+b_m\alpha_m]_W$.
From Proposition~\ref{c1} we obtain:
\[
\langle c_1(J),A\rangle = b_1 + \dots + b_m.
\]
On the other hand, since $A$ is represented by an $\omega'$-symplectic surface of genus $g$, it satisfies the adjunction formula:
\[
\langle c_1(J'),A\rangle - A\cdot A = 2 - 2g \le 0.
\]
Putting everything together:
\[
\sum (b_j+b_j^2) = b_1 + \dots + b_m + b_1^2+\dots + b_m^2 \le 0.
\]
However, each of the summands on the left-hand side is non-negative, and evaluates to $0$ only if $b_j \in \{-1,0\}$ for each $j$.
If $g>1$, we are already done.
If $g=1$, the signs of all coefficients of $b_j$ agree, and therefore, by a winding number argument, $[b_1\alpha_1+\dots+b_m\alpha_m]_P$ cannot be zero in $H_1(P)$, as desired.
\end{proof}

The previous theorem rules out the presence of symplectic surfaces of genus $g>0$. 
Symplectic spheres can exist in a weak symplectic filling of a planar contact structure, and we will  now describe their homology 
classes  explicitly 
in terms of vanishing cycles of a Lefschetz fibration deformation equivalent to the given minimal symplectic filling. 

Let us set up some notation and terminology first.
We say that two curves $\alpha$ and $\alpha'$ in $P\subset D^2$ are \emph{separated}
if there is no hole in $P$ around which both $\alpha$ and $\alpha'$ have positive winding number.
(Equivalently, this means that $\alpha$ and $\alpha'$ are homologous to $\beta = \partial D$ and $\beta' = \partial D'$ such that $D$ and $D'$ are disjoint.)
We say that $\alpha$ \emph{dominates} $\alpha'$, and we write $\alpha \succ \alpha'$, if there is no hole in $P$ around which the winding number of $\alpha'$ is larger than the winding number of $\alpha$.
(Equivalently, this means that $\alpha$ and $\alpha'$ are homologous to $\beta = \partial D$ and $\beta' = \partial D'$ such that $D$ contains $D'$.)
Note that $\succ$ is \emph{not} a partial order on isotopy classes of curves, but rather it induces one on homology classes of embedded curves.

\begin{lemma}\label{spheres-coefficients}
Suppose that  $(W, \omega')$ is a minimal weak filling of a planar contact manifold $(Y, \xi)$, deformation equivalent to a Stein filling  $(W, \omega)$ supported by a planar Lefschetz fibration with vanishing cycles $\alpha_1,\dots,\alpha_m\subset P$.
If there is an embedded symplectic sphere in $(W, \omega')$ in the homology class $[b_1\alpha_1+\dots+b_m\alpha_m]_W$, then all coefficients $b_i$ are either $0$ or $\pm1$, and there is exactly one coefficient $+1$.
Without loss of generality, suppose that the homology class is $[\alpha_1-\alpha_2-\dots-\alpha_\ell]_W$;
then $\alpha_1 \succ \alpha_j$ for every $j=2,\dots,\ell$, and $\alpha_j$ and $\alpha_{j'}$ are separated for every $j\neq j'$ among $2,\dots,\ell$.
\end{lemma}

\begin{proof}
The proof is immediate once we write the adjunction formula as in the previous proof; indeed, the equation
\[
\sum (b_j+b_j^2) = 2
\]
implies that all coefficients $b_j$ are either $0$ or $-1$, except for exactly one $j$, for which $b_j = -2$ or $b_j = 1$.
However, the first case is excluded, since otherwise all coefficients would have the same sign.

We now turn to the second part of the statement.
Since $[\alpha_1-\alpha_2-\dots-\alpha_\ell]_W$ is a homology class in $H_2(W)$, $[\alpha_1]_P = [\alpha_2]_P+\dots+[\alpha_\ell]_P$.
Fix a hole in $P$, and let us consider the winding number $w_i$ of $\alpha_i$ around it.
Since $[\alpha_1-\alpha_2-\dots-\alpha_\ell]_P=0\in H_1(P)$, its total winding number around the hole is $0$; on the other hand, by linearity, it is also $w_1-w_2-\dots-w_\ell$, so we have $w_1 = w_2+\dots+w_\ell$.
Since each $w_i$ is either 0 or 1, we immediately see that $\alpha_1\succ\alpha_2, \dots, \alpha_\ell$, and that $\alpha_i$ and $\alpha_j$ are separated whenever $2\le i\neq j\le \ell$.
\end{proof}

\section{Links of hypersurface singularities}\label{links}

In this section, we turn our attention to links of  isolated singularities of complex hypersurfaces in $\CC^3$.
Consider a complex hypersurface $\Sigma\subset \CC^3$, given by an
equation $F(z_1, z_2, z_3)=0$ with an isolated critical point at the origin, 
and let $(Y, \xi)$ be the link of the singularity with its canonical contact structure, so that 
 $Y= \Sigma \cap \{|z_1|^2+|z_2|^2+|z_3|^2 = \eps\}$.
The manifold $(Y, \xi)$ is Stein fillable, with the standard filling given by the {\em Milnor fiber} $\{F(z_1, z_2, z_3)=\eta \} \cap D^6$ 
for small $\eta>0$;
the Milnor fiber is the smoothing of $\Sigma \cap \{|z_1|^2+|z_2|^2+|z_3|^2 \le \eps\}$.

We now consider an example, the link of the $D_4$-singularity.
As a 3-manifold, this is described by the surgery diagram of Figure~\ref{seifert-space} where $e_0 = -1/r_1 = -1/r_2 = -1/r_3 = -2$;
that is, it is the boundary of the plumbing associated to the graph $D_4$ (see Figure~\ref{Dn}).
This example illustrates the main idea of our obstruction and will also be the key case of Theorem~\ref{sing}.

\begin{lemma}\label{d4}
The canonical contact structure on the link of the $D_4$-singularity is not planar.
\end{lemma}

\begin{proof} Consider the Milnor fiber $W$ of the $D_4$-singularity. This is a Stein filling of the canonical contact
structure on the link.
The intersection form of $W$ is given by the $D_4$-graph (Figure~\ref{Dn}).
We label its central vertex $X$, and the other vertices $A$, $B$, $C$.

For the sake of contradiction, suppose that the canonical contact structure on $D_4$ admits an open book with planar page $P$.
As before, by Wendl's theorem we know that $W$ admits the structure of a Lefschetz fibration whose fiber is the page $P$, 
and the vanishing cycles come from a positive factorization of the monodromy.
The intersection form on $W$ can be computed as in Proposition~\ref{iform}; 
we now examine possibilities for vanishing classes that could produce $D_4$.   

To begin, we need to have four classes with self-intersection $-2$.
By Proposition~\ref{iform}, each of these must be given by the difference of two curves,
corresponding to two \emph{distinct} vanishing cycles (which could, however, be isotopic as curves in $P$);
moreover, it must be a null-homologous linear combination, so the two curves should be homologous.
This means that the two curves must encircle the same holes of the disk.
Note that the curves {\em do not} have to bound an annulus and {\em do not} have to be homotopic, see Figure~\ref{homolo-curves}.

\begin{figure}[ht]
\includegraphics[scale=0.6, keepaspectratio]{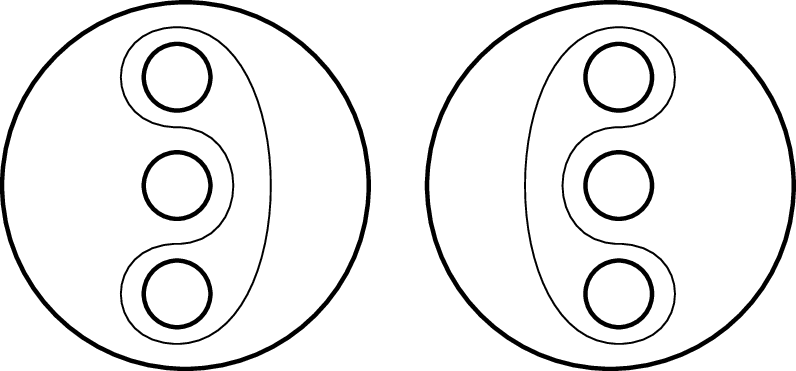}
\caption{Simple closed curves in $P$ are homologous if and only if they encircle the same holes. The curves shown are homologous but not homotopic in the three-holed disk.}
\label{homolo-curves}
\end{figure}

Let the class of the central vertex $X$ be $[\alpha-\beta]_W$.
Similarly, the class $A$ is given by two homologous curves, and since $A \cdot X = 1$, exactly one of these curves must coincide with $\alpha$ or $\beta$.
We may assume that $A = [\gamma-\alpha]_W$  (where
the vanishing cycle $\gamma$ is different from both $\alpha$ and $\beta$); note also that $\alpha$ and $\beta$ must be distinct.
Similarly, both classes $B$ and $C$ must be given by pairs of vanishing cycles, so that exactly one of the curves in the difference representing each pair coincides with $\alpha$ or $\beta$. 
However, since $A \cdot B =A \cdot C = B \cdot C = 0$, no curves may be used in more than one pair,
which is clearly not possible.
Indeed, if $B = [\beta-\delta]_W$, $C$ can use neither $\alpha$ nor $\beta$.
\end{proof}

We are now ready to prove Theorem~\ref{sing}.       
         
\begin{proof}[Proof of Theorem~\ref{sing}] Since  the $D_n$-graph contains $D_4$, the argument of Lemma~\ref{d4} applies to show that
the link of the $D_n$-singularity is not planar for any $n>4$. 
 \begin{figure}[ht]
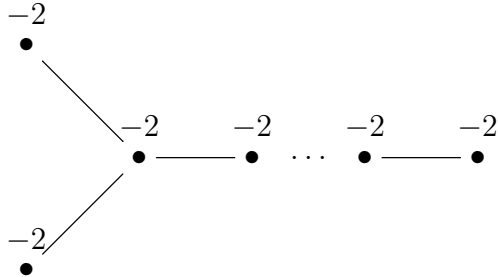

$
\xygraph{
!{<0cm,0cm>;<1cm,0cm>:<0cm,1cm>::}
!~-{@{-}@[|(2.5)]}
!{(0,0) }*+{\bullet}="a1"
!{(1.5,0) }*+{\bullet}="a2"
!{(2.5,0) }*+{\dots}="dots"
!{(5,0) }*+{\bullet}="ak1"
!{(3.5,0) }*+{\bullet}="ak"
!{(-1.06,1.06) }*+{\bullet}="c"
!{(-1.06,-1.06) }*+{\bullet}="d"
!{(0,0.4) }*+{-2}
!{(1.5,0.4) }*+{-2}
!{(5,0.4) }*+{-2}
!{(3.5,0.4) }*+{-2}
!{(-1.06,1.46) }*+{-2}
!{(-1.06,-1.46) }*+{-2}
"c"-"a1"
"a2"-"a1"
"ak1"-"ak"
"d"-"a1"
"a2"-"dots"
"dots"-"ak"
}$
  \caption{The $D_n$-graph, which has $n$ vertices, all labeled with $-2$.}
  \label{Dn}
\end{figure}
For links of all other surface singularities, the theorem follows from previously known results. 
Indeed, by~\cite{Tjurina}, 
the only surface singularities with negative definite Milnor fiber are 
the simple singularities $A_n$, $D_n$, $E_n$.
Etnyre's theorem says that every filling of a planar contact structure must be negative definite~\cite{Etn}, 
and since the Milnor fiber gives a Stein filling, it follows that only the links of A-D-E singularities can be planar.
The case of $E_8$ is ruled out by~\cite[Theorem 1.2]{Etn}, as the corresponding link is an integral homology sphere with a non-standard 
intersection form. 
The cases of $E_6$ and $E_7$ are similarly ruled out using~\cite[Theorem 1.2]{Etn}: although not stated explicitly in Etnyre's paper, 
the same proof applies to show that for a planar rational homology sphere, the intersection form of any Stein filling must
embed in a negative definite diagonal lattice.
 The links of $E_6$ and $E_7$ are rational homology spheres; the corresponding Milnor 
fibers, i.e. fillings given by the plumbing graphs, have intersection forms $E_6$ and $E_7$. Neither embeds into the standard 
lattice, thus the canonical structures on the links of $E_6$ and $E_7$ cannot be planar.

Alternatively, the cases of $E_6$ and $E_7$ follow from Lemma~\ref{d4}, as the $E_6$- and $E_7$-graphs both contain the $D_4$-graph.

Finally, the links of the $A_n$-singularities are the lens spaces $L(n+1,n)$, and their canonical contact structures are easily seen to be planar~\cite{Scho}.
\end{proof}

\section{The homological obstruction}\label{therest}

We now prove Theorem~\ref{general}; the argument is very similar to the proof of Lemma~\ref{d4}.
As in Section~\ref{compute}, we use the notation $\alpha_\bb$ to denote the linear combination $b_1\alpha_1+\dots+b_m\alpha_m$ of curves associated to the $m$-tuple $\bb = (b_1,\dots,b_m)$.
Moreover, given an $m$-tuple $\bb$, we call the set $\{i \mid b_i \neq 0\}$ the \emph{support} of $\bb$.
In other words, the support of $\bb$ is the set of curves used by $\alpha_\bb$.
By extension, we also call the same set the support of the associated homology class $[\alpha_\bb]_W$ (when this makes sense).

\begin{lemma}
Suppose $W$ is the minimal filling of the contact structure $(Y,\xi)$, associated to the factorization of the monodromy $\phi: P\to P$, where $P$ is planar.
Suppose $B_1, B_2 \in H_2(W)$ satisfy
\[
B_1\cdot B_1, B_2\cdot B_2 \in \{-2,-3\}, \quad B_1\cdot B_2 = 0.
\]
Then, $B_1$ and $B_2$ have disjoint support.
\end{lemma}

\begin{proof}
The proof is split into three cases:
\begin{enumerate}
\item $B_1\cdot B_1 = B_2 \cdot B_2 = -2$,
\item $B_1\cdot B_1 = -2$, $B_2 \cdot B_2 = -3$, and
\item $B_1\cdot B_1 = B_2 \cdot B_2 = -3$.
\end{enumerate}
(The case $B_1\cdot B_1 = -3$, $B_2 \cdot B_2 = -2$ is clearly symmetric to the second case, so we can omit it.)

To fix the notation, suppose that $W$ is associated to the factorisation of $\phi$ into Dehn twists along curves $\alpha_1,\dots,\alpha_m$.
We recall that, if a class in a minimal weak filling of a planar contact structure has self-intersection $-2$, then it is represented by the difference of two homologous curves.
Along the same lines, if a class as above has self-intersection $-3$, it is of the form $\pm[\alpha_i - \alpha_j - \alpha_k]_W$, where $[\alpha_i]_P = [\alpha_j]_P + [\alpha_k]_P$.
In particular, there are holes in $P$ around which both $\alpha_i$ and $\alpha_j$ (respectively, $\alpha_k$) have both winding number 1.
\begin{enumerate}
\item Without loss of generality, suppose that $B_1 = [\alpha_1 - \alpha_2]_W$, where $\alpha_1$ and $\alpha_2$ are homologous in $P$.
If the support of $B_2$ is not disjoint from that of $B_1$, then $B_2 = [\alpha_i-\alpha_j]_W$ where $i$ or $j$ is either 1 or 2, 
and $[\alpha_i]_P$ and $[\alpha_j]_P$.
One easily sees that neither combination works: for example, if $B_2=[\alpha_3-\alpha_1]_W$ with $\alpha_3\neq \alpha_1, \alpha_2$, 
then $B_1 \cdot B_2=1$, if $B_2=[\alpha_2-\alpha_1]_W$, then $B_1 \cdot B_2=2$, and other cases are similar.

\item As above, suppose $B_1 = [\alpha_1 - \alpha_2]_W$.
If the support of $B_2$ is not disjoint from that of $B_1$, the only possibility is that 
$\pm B_2 = [\alpha_3 - \alpha_1 - \alpha_2]_W$, since $\alpha_1$ and $\alpha_2$ must appear with the same sign.
But $\alpha_1$ and $\alpha_2$ have winding number 1 around the same holes, and so their sum cannot be homologous to a simple closed curve $\alpha_3$.

\item Without loss of generality, suppose that $B_1 = [\alpha_1 - \alpha_2 - \alpha_3]_W$.
(Note that here we are using that the assumptions are unchanged if we change sign to either $B_1$ or $B_2$.)
Suppose that the support of $B_2$ is not disjoint from that of $B_1$.
Up to relabeling the indices and up to changing the sign of $B_2$, the only possibility is that $B_2 = [\alpha_4 - \alpha_1 - \alpha_2]_W$.
But, again, as observed above, $\alpha_1$ and $\alpha_2$ have winding number 1 around \emph{some} hole, and so their sum cannot be homologous to a simple closed curve $\alpha_4$.\qedhere
\end{enumerate}
\end{proof}

\begin{proof}[Proof of Theorem~\ref{general}]
Suppose we have a configuration of curves $B_1,\dots, B_k$, $X$ as in the statement.
By the previous lemma, $B_1, \dots, B_k$ have pairwise disjoint supports.
Since $X$ meets non-trivially each of $B_1,\dots,B_k$, its support must intersect at least the support of each of them, and in particular $X\cdot X \le -k$, thus leading to a contradiction.
\end{proof}

\begin{proof}[Proof of Corollary~\ref{seifert}]
For Seifert fibered L-spaces  $M(-2; r_1, r_2, r_3)$, $r_i \in (0,1) \cap \QQ$, classification of tight contact structures was given in~\cite{Gh}.
Every tight contact structure on this space can be obtained by expanding 
the rational parameters $-\frac1{r_i}$ as continued fractions,
\begin{equation*}
-\frac1{r_i}=a_0^{i}-\cfrac{1}{a_1^{i}-\cfrac{1}{a_2^{i} - \cdots}}
\end{equation*}
and making a Legendrian surgery diagram where each $-\frac1{r_i}$-framed circle is replaced by a chain of Legendrian unknots with Thurston--Bennequin numbers 
given by the coefficients $a_0^{i}+1, a_1^{i}+1, \dots$.
Since by assumption $r_1, r_2, r_3 \geq \frac13$, we have that
$a_0^{1}, a_0^{2}, a_0^{3} \in \{-2, -3\}$.
Thus, the corresponding plumbing graph for the Stein filling satisfies the hypotheses of Theorem~\ref{general}, and therefore the contact structure is not planar.
\end{proof}
         
We observe that Theorem~\ref{general} applies in many situations where the filling is not a plumbing of spheres.

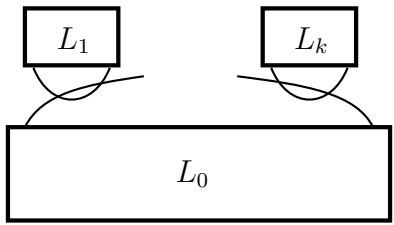
\begin{figure}[ht]
\labellist
\pinlabel $L_0$ at 88 23
\pinlabel $L_1$ at 32 86
\pinlabel $L_k$ at 144 86
\pinlabel $\dots$ at 89 86
\endlabellist
\includegraphics[scale=0.8]{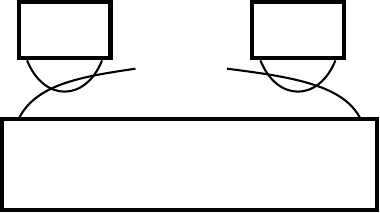}
\caption{The Legendrian surgery diagram for Example~\ref{e:paolo}.
The Legendrian knot $L_i$ have $k-1 < \tb(L_0) < 0$, $\tb(L_1), \dots, \tb(L_{k}) \in \{-1, -2\}$.}\label{f:paolo}
\end{figure}

\begin{example}\label{e:paolo}
Consider the Legendrian surgery diagram of Figure~\ref{f:paolo}, where $k > 2$ and the Legendrian tangles $L_0, \dots, L_k$ satisfy the following properties:
\begin{itemize}
\item the closure of $L_0$ is a Legendrian knot with $1-k < \tb(L_0) < 0$;
\item for each $i > 0$ the closure of $L_i$ is a Legendrian link; the closure of the \emph{arc} in the tangle is a Legendrian knot $L^0_i$ with $\tb(L^0_i)\in\{-1,-2\}$ for each $i$;
\item at least one of the knots $L_0, L_1^0, \dots, L_k^0$ is nontrivial.
\end{itemize}
Then the corresponding Stein 4-manifold $(W,J)$ is not diffeomorphic to a plumbing of spheres, but it contains a configuration of homology classes that satisfy the assumptions of Theorem~\ref{general}.
Therefore, the boundary $(Y,\xi)$ of $(W,J)$ is not a planar contact manifold.
\end{example}

Finally, we observe that the technique of the proof of Theorem~\ref{general} has a limit.
Indeed, the configuration below gives no obstruction to planarity simply by looking at intersection forms, as the intersection form given by the graph in the figure can be embedded into the intersection form of the unique symplectic filling of $L(6,5)$.
The filling $W$ is associated to the open book $(S^1\times I, \tau^6)$,  where $\tau$ is the right-handed Dehn twist along the core of the annulus, and the underlying 4-manifold of the filling is diffeomorphic to a linear plumbing of five spheres of self-intersection $-2$.
There are six vanishing cycles, $\alpha_1,\dots,\alpha_6$, all parallel to the core of the annulus;
the homology classes of the spheres in the plumbing are $[\alpha_1-\alpha_2]_W, \dots, [\alpha_5-\alpha_6]_W$.
The embedding we seek is then given as follows (here we omit the suffix for the homology classes for readability):
\[
\xygraph{
!{<0cm,0cm>;<1cm,0cm>:<0cm,1cm>::}
!~-{@{-}@[|(2.5)]}
!{(0,0) }*+{\bullet}="a1"
!{(1.5,0) }*+{\bullet}="a2"
!{(-1.06,1.06) }*+{\bullet}="c"
!{(-1.06,-1.06) }*+{\bullet}="d"
!{(0,0.4) }*+{-2}
!{(1.5,0.4) }*+{-2}
!{(-1.06,1.46) }*+{-4}
!{(-1.06,-1.46) }*+{-2}
"c"-"a1"
"a2"-"a1"
"d"-"a1"
} \qquad = \qquad 
\xygraph{
!{<0cm,0cm>;<1cm,0cm>:<0cm,1cm>::}
!~-{@{-}@[|(2.5)]}
!{(0,0) }*+{\bullet}="a1"
!{(1.5,0) }*+{\bullet}="a2"
!{(-1.06,1.06) }*+{\bullet}="c"
!{(-1.06,-1.06) }*+{\bullet}="d"
!{(0.4,0.4) }*+{\phantom{}_{[\alpha_2 - \alpha_3]}}
!{(1.8,0.4) }*+{\phantom{}_{[\alpha_3 - \alpha_4]}}
!{(-1.06,1.46) }*+{\phantom{}_{[\alpha_1+\alpha_2-\alpha_5-\alpha_6]}}
!{(-1.06,-1.40) }*+{\phantom{}_{[\alpha_1-\alpha_2]}}
"c"-"a1"
"a2"-"a1"
"d"-"a1"
}
\]
The boundary of the plumbing given above is the Seifert manifold $M\! \left (-2; \frac 12, \frac 12, \frac 14 \right )$.
It has three tight contact structures, each presented by a contact surgery diagram directly coming from the plumbing graph~\cite{Tos} (see also~\cite{Gh,LS3}).
Theorem~\ref{general} gives no information on planarity of these contact structures; however, Theorem~\ref{sing2} shows that two of them (one conjugate to the other) are not planar.
We do not know whether the third one (which is self-conjugate) is planar or not.

\section{Links of normal surface singularities}\label{normal-links}

The goal of this section is to prove Theorem~\ref{sing2}.
We will use terminology and results from Section~\ref{compute} without explicit mention; before proving the theorem, we recall a few facts and definitions concerning resolutions of surface singularities and plumbing graphs.

Given a complex surface $X$ with an isolated  singularity at $0$, we can consider its resolution $\pi:\tilde{X}\to X$.
The resolution is {\em good} if the irreducible components of the exceptional divisor $\pi^{-1}(0)$ are smooth complex curves that intersect transversely at double points only.
The topology of the resolution is encoded by the {\em (dual) resolution graph} $\Gamma$. The vertices of $\Gamma$ correspond to irreducible components of the exceptional divisor  and are labeled by the genus and the self-intersection (weight) of the corresponding curve; the edges record intersections of different irreducible components.
The link of the singularity is then the boundary of the plumbing of disk bundles over surfaces according to $\Gamma$. 
A good resolution is not unique, but graphs arising from different resolutions are related by a finite sequence of blow-ups/blow-downs of vertices corresponding to spheres with self-intersection $-1$.

It is known that a surface singularity is normal if and only if its resolution graph is negative-definite.
(This property simultaneously holds or fails for resolution graphs of all good resolutions.)  
For a graph with negative integer weights associated to its vertices, recall that a {\em bad vertex} is a vertex $v$ with weight $-w(v)$ such that 
\[
0<w(v)< \text{valence of $v$.}
\] 
A normal surface singularity is known to be rational if its graph has no bad vertices, but the converse is not true.  
(We refer the reader to~\cite{Nem5} for details of the definitions above and their topological significance.) We will need a subclass of rational singularities:
\begin{definition}\label{def-redfundc}
A normal surface  singularity whose dual resolution graph is a tree of spheres with no bad vertices is called a {\em rational singularity with reduced fundamental cycle.}   
\end{definition}
In this paper we  only work with resolution graphs, without referring to any other properties of this class of singularities. N\'emethi~\cite{NemL} proved that a normal surface singularity is rational if and only if its link is an L-space.
Using this, Theorem~\ref{sing2} implies that if the canonical contact structure on a link of singularity is planar, then this link must be an L-space.     

As mentioned in the introduction, canonical contact structures are known to be planar for links of rational surface singularities with reduced fundamental cycle.
We now prove the converse: if the canonical contact structure admits a planar open book, then the resolution graph is a tree of spheres with no bad vertices. 
We begin with two lemmas; both are probably well-known, but for convenience we give their proofs.

\begin{lemma}\label{l:minimalGamma}
Suppose the graph $\Gamma$ is minimal, i.e. it contains no vertices of weight $-1$ given by spheres.
Then $(W, \omega')$ is minimal.
\end{lemma}

\begin{proof}
The assumption on $\Gamma$ translates as: for each vertex $E_i$, either the weight $w_i = E_i\cdot E_i$ 
satisfies $w_i < -1$, or $w_i = -1$ and $g(E_i) > 0$.

We claim that no homology class $E = \sum_i a_iE_i$ on which $\omega$ 
integrates positively can satisfy $E\cdot E = -1$ and $g(E) = 0$; in fact, 
such a class would also have to satisfy $c_1(E) = 1$, by adjunction.

We first claim that all the coefficients $a_i$ are all positive.
Since any class with all negative coefficients  obviously cannot be symplectic, 
it is enough to show that all coefficients must have the same sign.
First, we observe that the class $E$ is indecomposable, i.e. if we write $E = E'+E''$, where $E'\cdot E'' = 0$, then either $E'=0$ or $E''=0$; in fact, since $\Gamma$ is negative definite, if $E$ decomposed as $E'+E''$, both $E'$ and $E''$ would have negative square, and $E\cdot E = E'\cdot E' + E''\cdot E'' \le -2$.
This, in turn, implies that the support of $E$, i.e. the set of vertices for which $a_i \neq 0$, is connected.
Suppose now that the coefficients do not all have the same sign; then there are two coefficients $a_i < 0 < a_j$ such that $E_i \cdot E_j = 1$.
Write $|E|$ for the homology class $|E| = \sum_i |a_i|E_i$.
We now observe that
\[
-1 = E\cdot E = \sum_i a_i^2 w_i + \sum_{i,j} a_ia_j E_i\cdot E_j < \sum_i a_i^2 w_i + \sum_{i,j} |a_i||a_j|E_i\cdot E_j = |E|\cdot|E| \le -1.
\]

Now we know that $a_i \ge 0$ for each $i$.
Adjunction for each vertex shows that $c_1(E_i) = 2 - 2g(E_i) + w_i$, and the latter quantity is never positive by assumption. Hence,
\[
1 = c_1(E) = \sum_i a_i c_1(E_i) \le 0.\qedhere
\]
\end{proof}

\begin{lemma}\label{l:genus-curve-sing}
Let $(C,0)$ be a curve singularity in $(\CC^2,0)$. The Milnor fiber of $C$ has genus $0$ if and only if $(C,0)$ is either smooth or a double point.
\end{lemma}

\begin{proof}
Indeed, if $g$ is the genus of the Milnor fiber of $(C,0)$, $\mu$ its Milnor number, $r$ its number of branches, and $\delta$ its delta-invariant, then $g = 1 + \delta - r$; since the multiplicity of $(C,0)$ is at least $r$, then $\delta \geq r(r-1)/2$, $g$ is positive unless the singular point is smooth or an ordinary double point.
(See, for instance,~\cite[Pages 572--574]{BrieskornKnorrer} for more details.)
\end{proof}

We now turn to the proof of Theorem~\ref{sing2}.
The strategy is the following.
Suppose we have a normal surface singularity, and let $\Gamma$ be the graph associated to its smallest good resolution.
If $\Gamma$ is not minimal, then we blow down to the minimal graph; we will show that this either has a singular curve or a point of higher intersection (e.g. a tangency or a singular point with more than two branches).
In either of the two cases, we can construct a divisor by smoothing (some of) the singularities, and this divisor will have positive genus.
If, on the other hand, $\Gamma$ is minimal, we will apply Wendl's theorem, and argue that there can be no vertices with higher genus, nor cycles in the graph, nor bad vertices (in a way similar to the proof of~\ref{general}).

\begin{proof}[Proof of Theorem~\ref{sing2}]
The canonical contact structure of the link of a normal surface singularity $(X, 0) \subset \CC^N$ has a symplectic filling given by a good resolution $\pi: \tilde{X} \to X$.
Note that $\tilde{X}$ lives in a blowup of $\CC^N$, hence it is K\"ahler, and in particular it has a symplectic form $\omega'$; 
the preimage $\pi^{-1}(0)$ is a complex divisor, and in particular it is symplectic.
More precisely, $Y= X \cap S^{2N-1}_{\eps}$ is filled by $W=\pi^{-1}(D^{2N}_{\eps})$, with the (restriction of the) symplectic structure $\omega'$.
The irreducible components of the exceptional divisor are then symplectic surfaces in $\tilde X$, so that $(Y, \xi)$  
is the convex boundary of a plumbing of symplectic surfaces; as in the introduction, the plumbing is encoded by the 
resolution graph $\Gamma$.  

We would like to use Wendl's theorem and arguments with vanishing cycles as before; however, the filling $(W, \omega')$ is not necessarily minimal, and we have to perform some blow-downs before a compatible Lefschetz fibration can be found. 
Reduction to the case of minimal fillings is done as follows.
If $(W, \omega')$ is not minimal, i.e. it contains a symplectic sphere $E$ with $E\cdot E=-1$, 
we use  Lemma~\ref{l:minimalGamma} to find a vertex of genus 0 and weight $-1$ in the graph $\Gamma$.
Suppose now that the graph $\Gamma$ contains vertices of genus 0 and weight $-1$. 
We blow down the corresponding divisors until we get a minimal graph.
The corresponding resolution may no longer be good; there may be singular curves among the components of 
the exceptional divisor or multiple intersection points.
We can smooth out the singular curve by replacing each singular point (of the curve in a surface) by its Milnor fiber.
Similarly, if there are intersection points of multiplicity greater than 2 or tangencies,
we also smooth them out (as a reducible singularity).
This process creates a divisor of positive genus, because, thanks to Lemma~\ref{l:genus-curve-sing}, the Milnor fibre of a curve singularity is planar if and only if we have a smooth or a double point.
As a result, we found a symplectic surface of positive genus in a symplectic filling of $(Y, \xi)$, so by Theorem~\ref{positivegenus}, $(Y, \xi)$ cannot be planar in this case.   

It remains to prove the statement of the theorem for the case where a good resolution $\tilde{X}$ is also minimal, i.e. $(W, \omega')$ contains no spheres of self-intersection $-1$. 
Obviously, in this case $\Gamma$ contains no vertices corresponding to spheres with weight $-1$. The minimal weak symplectic filling $(W, \omega')$ is deformation equivalent to a Lefschetz fibration, and we can use the results of Section~\ref{compute}.
We need to prove the following three facts, for $(Y,\xi)$ planar:

\begin{itemize}
\item[(g)] all surfaces in the plumbing have genus 0;
\item[(c)] there are no cycles in the graph;
\item[(b)] the graph has no bad vertices.
\end{itemize}

Condition (g) is guaranteed by Theorem~\ref{positivegenus}, so we only need to prove conditions (c) and (b).

\begin{itemize}
\item[(c)] Suppose that there is a cycle; namely, that there are classes $A_1,\dots, A_{c+1} = A_1$, such that 
$A_k \cdot A_{k+1} = 1$ for every $k = 1,\dots c$.
The divisor $A_1 \cup \dots \cup A_c$ is represented by symplectic spheres with positive, transverse intersections;
smoothing all intersections, we obtain a symplectic torus, which contradicts Theorem~\ref{positivegenus}.

\item[(b)] 
The filling $(W, \omega')$ is deformation equivalent to a Lefschetz fibration with the planar fiber;
write $\alpha_1, \dots \alpha_m$ 
for its vanishing cycles. 
The argument is now  similar to the one in the proof of Theorem~\ref{general}. 
Suppose $B$ is a vertex of the graph with $B\cdot B = \ell$, connected to vertices $A_1,\dots, A_n$.
Lemma~\ref{spheres-coefficients} describes the class of $B$;
without loss of generality, assume that $B = [\alpha_1 - \alpha_2 - \dots - \alpha_\ell]_W$, so that $\alpha_1$ dominates the curves $\alpha_2, \dots, \alpha_{\ell}$.

Again by Lemma~\ref{spheres-coefficients}, we know that there exist indices $i_k$ and collections of indices $J_k\not\ni i_k$ such that 
$A_k = [\alpha_{i_k} - \sum_{j\in J_k} \alpha_j]_W$, with $\alpha_{i_k} \succ \alpha_j$ for every $j \in J_k$.
For convenience, let $J_B = \{2,\dots,\ell\}$.
We can then compute:
\begin{align}
0 &= A_k\cdot A_{k'} = \delta(i_k,J_{k'}) + \delta(i_{k'},J_k) - \delta(i_k,i_{k'}) - \#(J_k \cap J_{k'}), \label{e:AA}\\
1 &= B\cdot A_{k} = \delta(1,J_{k}) + \delta(i_{k},J_B) - \delta(1,i_{k}) - \#(J_B \cap J_{k}) \label{e:AB}
\end{align}
where $\delta(i,J) = 1$ if $i\in J$, and is 0 otherwise, and $\delta(i,i') = 1$ if $i = i'$, and is 0 otherwise.

Let us focus on~\eqref{e:AB} first.
From it, we deduce that at least one among $1 \in J_{k}$ and $i_{k} \in J_B$ holds.
Suppose that both hold simultaneously;
then $\alpha_{1} \succ \alpha_{i_{k}}$ and $\alpha_{i_{k}} \succ \alpha_{1}$, 
which implies that the two curves $\alpha_{1}$ and $\alpha_{i_{k}}$ are homologous, 
and therefore that $B = [\alpha_1 - \alpha_{i_k}]_W$ and $A_k = [\alpha_{i_k} - \alpha_{1}]_W$, 
which clearly contradicts the assumption that $B\cdot A_k = 1$.

Next, we claim that $1 \in J_{k}$ can only hold for at most one of the classes 
$A_k = [\alpha_{i_k} - \sum_{j\in J_k} \alpha_j]_W$. Indeed, suppose that we have $1 \in J_{k}$ and $1 \in J_{k'}$ for two 
distinct classes $A_k$, $A_{k'}$. This implies that both leading terms $\alpha_{i_k}$ and $\alpha_{i_{k'}}$ dominate $\alpha_1$. 
Then we must have 
$\alpha_{i_k} \not\in J_{k'}$, because otherwise $\alpha_{i_k}$ would be separated from $\alpha_1$, and 
similarly $\alpha_{i_{k'}} \not\in J_{k}$. It follows that $A_k\cdot A_{k'} \leq -1$, a contradiction.

Finally, we want to show that $i_k \neq i_{k'}$ for every pair $k,k'$;
to this end, we use~\eqref{e:AA}.
Suppose that there are two indices, $k, k'$, such that $i_k = i_{k'} = i$;
this implies that $i\not\in J_k, J_{k'}$, and that $\delta(i_k, i_{k'}) = 1$.
But then $0 = A_k \cdot A_{k'} \le -1$, clearly a contradiction.

Summarizing, we see that the set $J_B=\{2, \dots, \ell\}$ must contain all the leading elements $i_1, \dots i_n$ 
of the classes $A_1, \dots, A_n$, except possibly one. Since  $i_1, \dots i_n$ are all distinct, it follows that 
that $n \le \ell$, i.e. that $B$ is not a bad vertex.
\qedhere
\end{itemize}
\end{proof}

We conclude this section with the proof of Corollary~\ref{deformation}.

\begin{proof}[Proof of Corollary~\ref{deformation}]
Let $(Y,\xi)$ be a planar contact 3-manifold.
The proof of Theorem~\ref{sing2} above shows that no filling of $(Y,\xi)$ can contain the exceptional divisor of a resolution of a non-planar singularity.
If there were a strong symplectic cobordism $(W',\omega')$ from the link $(Y_s,\xi_s)$ of a non-planar normal surface singularity to $(Y,\xi)$, then one could glue the resolution of $(Y_s,\xi_s)$ to obtain a strong symplectic filling $(W,\omega)$ of $(Y,\xi)$ containing a forbidden configuration.

The second half of the statement is now straightforward, since a deformation from $(S,0)$ to $(S',0)$ gives rise to a Weinstein cobordism from the link of $(S',0)$ to the link of $(S,0)$.
\end{proof}

\section{Planar Lefschetz fibrations with prescribed fundamental group}\label{pi1}

We will now construct planar Lefschetz fibrations with prescribed fundamental group.
Recall that the \emph{deficiency} of a presentation $\langle x_1,\dots,x_m \mid r_1, \dots, r_n \rangle$ is $m -n$, and that the \emph{deficiency} of a finitely presented group is the maximal deficiency over all its presentations.
A group is {\em perfect} if its abelianization is trivial.
 
\begin{prop}\label{p:ballgroups}
Let $G$ be a finitely presented group. Then there exists a planar Lefschetz fibration on a $4$-manifold $W$ 
with fundamental group $G$.
Moreover, if $G$ is perfect and has deficiency $0$, $W$ can be chosen to be an integral homology Stein $4$-ball. In this case, 
$\d W$ is an integral homology $3$-sphere.
\end{prop}

The family of perfect, finitely presented groups of deficiency 0 is quite rich: for instance, it contains fundamental groups of integral homology spheres.
In fact, let $G = \pi_1(Y)$ be the fundamental group of an integral homology 3-sphere $Y$; $G$ is perfect since its abelianization is $H_1(Y) = 0$.
Moreover, $G$ has non-positive deficiency, since it is perfect;
it has non-negative deficiency since a genus-$g$ Heegaard decomposition of $Y$ gives a presentation of $G$ with $g$ generators and $g$ relators (which is, in particular, a finite presentation).

The proof of Proposition~\ref{p:ballgroups} easily follows from the following lemma.

\begin{lemma}\label{l:goodpresentation}
Let $\langle x_1,\dots,x_m \mid r_1, \dots, r_{m-d} \rangle$ be a presentation of a group $G$ of deficiency $d$. 
Then there exists another presentation $\langle y_1,\dots,y_n \mid s_1, \dots, s_{n-d} \rangle$ of $G$ such that:
\begin{itemize}
\item[(p)] each word $s_j$ is a positive word in $y_1, \dots, y_n$;
\item[(r)] each generator $y_i$ appears at most once in each word $s_1,\dots,s_{n-d}$;
\item[(c)] the cyclic order of the generators $y_i$ is preserved in each word $s_j$.
\end{itemize}
\end{lemma}

\begin{proof}[Proof of Proposition~\ref{p:ballgroups}]
By Lemma~\ref{l:goodpresentation}, $G$ has a presentation $\langle y_1,\dots,y_n \mid s_1, \dots, s_{n-d} \rangle$ with the properties (p), (r), and (c) as above.
Consider the $n$-holed disk $P$, with fundamental group $\pi_1(P) = \langle y_1, \dots, y_n\rangle$. We assume that the generators 
$y_1, y_2, \dots, y_n$ are given by loops going around one hole each, as in Figure~\ref{f:pi1gens}.
By the properties (p), (r), and (c), each word $s_j$ is represented by an embedded simple closed curve $\alpha_j$ in $P$. Indeed, 
since by (p) and (r) each generator enters in the word $s_j$ positively and at most once, 
we can take the curve $\alpha_j$ enclosing the corresponding holes, 
with a counterclockwise orientation. By (c), the cyclic order of the generators in the loop given by $\alpha_j$ is the same 
as in the word $s_j$; it follows that $\alpha_j$ represents $s_j$.
See Figure~\ref{f:allowedword} for an example.
Let $\phi$ be the product of positive Dehn twists along $\alpha_1,\dots,\alpha_{n-d}$.
By construction, the associated Lefschetz fibration $W$ is a Stein domain whose fundamental group is precisely $G$.

If $G$ is perfect, $H_1(W)=G/G'$ vanishes; if, moreover, $G$ has deficiency 0, using a presentation with $d=0$ yields $H_2(W)=0$, since the classes $\alpha_1, \dots, \alpha_n$ are linearly independent in $H_1(P)$.
\end{proof}

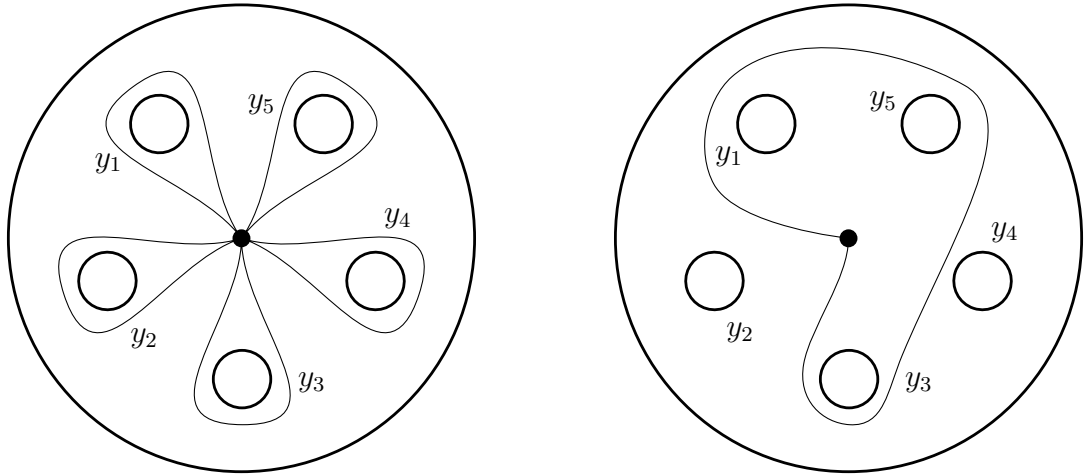
\begin{figure}
\begin{subfigure}[c]{0.4\textwidth}
    \labellist
    \pinlabel $y_1$ at 91 277
    \pinlabel $y_2$ at 123 120
    \pinlabel $y_3$ at 273 80
    \pinlabel $y_4$ at 350 227
    \pinlabel $y_5$ at 228 329
    \endlabellist
        \includegraphics[width=\textwidth]{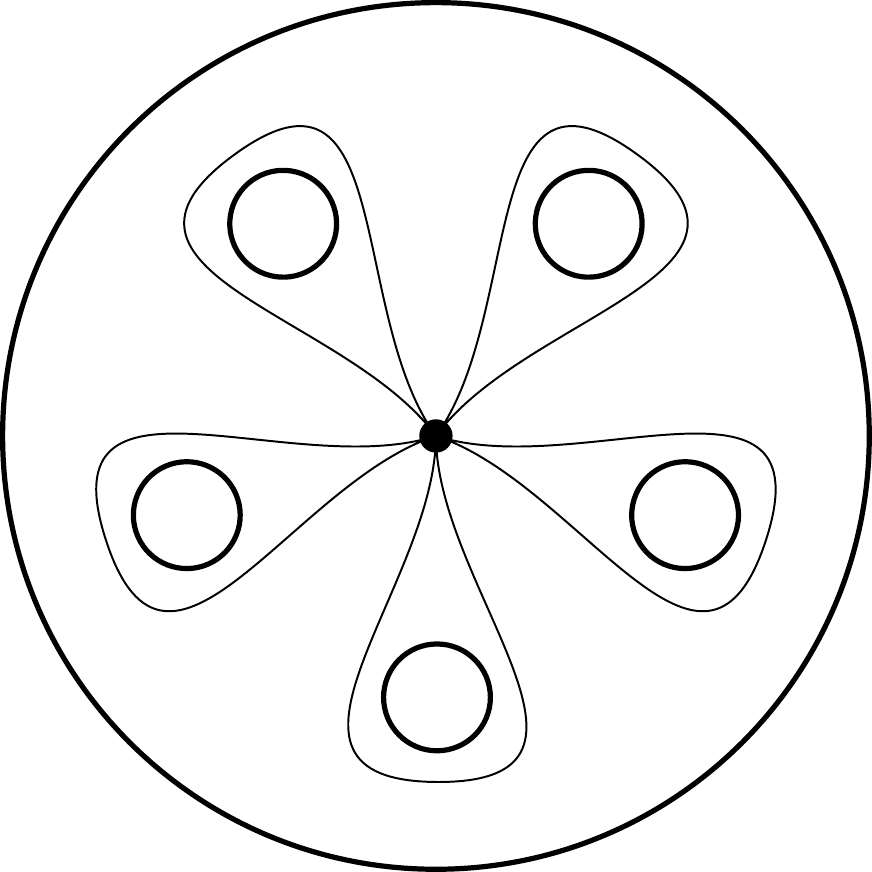}
        \caption{The chosen generators of the fundamental group of the 5-holed disk, with basepoint in the center.}
        \label{f:pi1gens}
\end{subfigure}
\hspace{0.1\textwidth}
\begin{subfigure}[c]{0.4\textwidth}
    \labellist
    \pinlabel $y_1$ at 103 283
    \pinlabel $y_2$ at 113 125
    \pinlabel $y_3$ at 273 80
    \pinlabel $y_4$ at 350 215
    \pinlabel $y_5$ at 241 332
    \endlabellist
        \includegraphics[width=\textwidth]{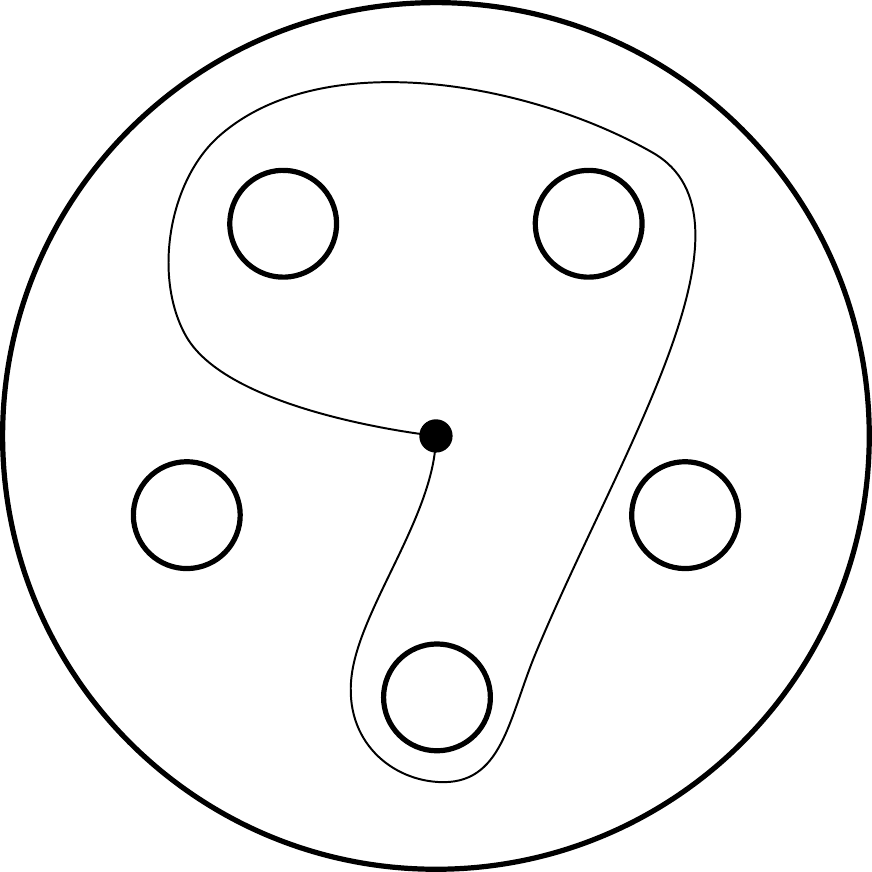}
        \caption{A simple closed curve representative for the word $y_3y_5y_1$.\\}
        \label{f:allowedword}
\end{subfigure}
\caption{Generators and a simple closed curve in the 5-holed disk.}\label{f:prc}
\end{figure}

Before proving the lemma, let us introduce the concept of badness for a presentation.
We say that a word is \emph{long} if its length is at least $3$, and
\emph{short} if it is of length $2$.
Given a presentation $\mathcal{P} = \langle a_1,\dots,a_m \mid w_1, \dots, w_{m-d} \rangle$,  we define its \emph{badness} $b(\mathcal P)$ as follows.
Let $b_-(\mathcal{P})$ be the sum of the number of occurrences of $a_i^{-1}$ over all generators $a_i$;
let also $b^i_+(\mathcal{P})$ be the sum of the number of occurrences of $a_i$ in long words,
and $b_+(\PP) = \sum_i\max\{b^i_+(\PP)-1,0\}$ (that is, we are ignoring the first appearance of each $a_i$ in long words, if there is one, as well as all appearances of $a_i$ in short words).
Let $b(\PP) = b_-(\PP)+b_+(\PP)$.

For instance, consider the presentation $\PP = \langle a,b,c,d \mid bad, cab, ab\inv ac\inv\rangle$; then we have $b_-(\PP) = 2$, $b_+^a(\PP) = 4$, $b_+^b(\PP) = 2$, $b_+^c(\PP) =b_+^d(\PP) = 1$, and hence $b(\PP) = 6$.

A key feature of $b$ that will be used in the proof is that it is insensitive to the labelling of the generators, in the sense that it is invariant under permutation of the indices of generators.

\begin{proof}[Proof of Lemma~\ref{l:goodpresentation}]
Notice that a presentation of badness 0 satisfies properties (p), (r), and (c),
up to reordering the generators.
(The converse, however, is not true.)
In fact, short words respect all cyclic orders, and, when the badness is 0, each generator appears in at most one long word; hence each long word can be used to define a compatible order on the corresponding subset of generators.

To prove the lemma, it is enough to show that, given a presentation $\mathcal{P}$ of deficiency $d$ and positive badness, we can always find another presentation $\mathcal{P'}$ for the same group with the same deficiency and with $b(\mathcal{P}') < b(\mathcal{P})$.

There are three cases to consider. Either the inverse of a generator appears, or a generator appears more than once somewhere in the presentation.
Without loss of generality, assume that that this generator is $a_1$, and let $w$ be one of the culprit words.

In the first case, the presentation $\mathcal{P}'$ is obtained from $\mathcal{P}$ by adding a generator $a_{m+1}$ and the relation $a_1a_{m+1}$, so that $a_{m+1} = a_1\inv$; we then replace one occurrence of $a_1\inv$ in $w_1$ by $a_{m+1}$.
We have replaced one occurrence of $a_1\inv$ with a new generator, so $b_-(\PP') = b_-(\PP)-1$, 
and we created a positive short word, so $b_+(\PP') = b_+(\PP)$.

In the second case, we add two generators $a_{m+1}, a_{m+2}$ and the relations $a_1a_{m+1}$, $a_{m+1}a_{m+2}$, 
so that $a_{m+2} = a_{m+1}\inv = a_1$; we then replace one occurrence of $a_1$ in $w$ by $a_{m+2}$.
We have replaced one extra occurrence of $a_1$ with a new generator, and created two positive short words, 
so $b_-(\PP') = b_-(\PP)$ and $b_+(\PP') = b_+(\PP)-1$.

In either case, $b(\PP') = b(\PP)-1$, and this concludes the proof of the lemma.
\end{proof}

In fact, one can extract a bound on the Euler characteristic of the page $P$ in terms of the original presentation $\mathcal{P}$ for $G$: the algorithm above gives a page $P$ with $\chi(P) \ge 1-n-2b_+(\mathcal{P})-b_-(\mathcal{P})$.

\bibliographystyle{amsplain}
\bibliography{nonplanar}

\providecommand{\bysame}{\leavevmode\hbox to3em{\hrulefill}\thinspace}
\providecommand{\MR}{\relax\ifhmode\unskip\space\fi MR }
\providecommand{\MRhref}[2]{%
  \href{http://www.ams.org/mathscinet-getitem?mr=#1}{#2}
}
\providecommand{\href}[2]{#2}
\begin{thebibliography}{10}

\bibitem{ABKP}
J.~Amor\'os, F.~Bogomolov, L.~Katzarkov, and T.~Pantev, \emph{Symplectic
  {L}efschetz fibrations with arbitrary fundamental groups}, J. Differential
  Geom. \textbf{54} (2000), no.~3, 489--545.

\bibitem{Aur}
Denis Auroux, \emph{A stable classification of {L}efschetz fibrations}, Geom.
  Topol. \textbf{9} (2005), 203--217.

\bibitem{Bu?}
Mohan Bhupal and Burak Ozbagci, \emph{Milnor open books of links of some
  rational surface singularities}, Pacific J. Math. \textbf{254} (2011), no.~1,
  47--65.

\bibitem{BrieskornKnorrer}
Egbert Brieskorn and Horst Kn{\"o}rrer, \emph{Plane algebraic curves}, Springer
  Science \& Business Media, 2012.

\bibitem{CNPP}
Cl\'{e}ment Caubel, Andr\'{a}s N\'{e}methi, and Patrick Popescu-Pampu,
  \emph{Milnor open books and {M}ilnor fillable contact 3-manifolds}, Topology
  \textbf{45} (2006), no.~3, 673--689. \MR{2218761}

\bibitem{djvs}
Theo de~Jong and Duco van Straten, \emph{Deformation theory of sandwiched
  singularities}, Duke Math. J. \textbf{95} (1998), no.~3, 451--522.

\bibitem{Eliashberg}
Yakov Eliashberg, \emph{A few remarks about symplectic fillings}, Geom. Topol.
  \textbf{8} (2004), 277--293.

\bibitem{Etnyre}
John~B. Etnyre, \emph{On symplectic fillings}, Algebr. Geom. Topol. \textbf{4}
  (2004), 73--80.

\bibitem{Etn}
\bysame, \emph{Planar open book decompositions and contact structures}, Int.
  Math. Res. Not. (2004), no.~79, 4255--4267.

\bibitem{GayMark}
David Gay and Thomas~E. Mark, \emph{Convex plumbings and {L}efschetz
  fibrations}, J. Symplectic Geom. \textbf{11} (2013), no.~3, 363--375.

\bibitem{GayStipsicz}
David~T. Gay and Andr\'as~I. Stipsicz, \emph{On symplectic caps}, Perspectives
  in analysis, geometry, and topology, Progr. Math., vol. 296,
  Birkh\"auser/Springer, New York, 2012, pp.~199--212.

\bibitem{Gh}
Paolo Ghiggini, \emph{On tight contact structures with negative maximal
  twisting number on small {S}eifert manifolds}, Algebr. Geom. Topol.
  \textbf{8} (2008), no.~1, 381--396.

\bibitem{Gi}
Emmanuel Giroux, \emph{G\'eom\'etrie de contact: de la dimension trois vers les
  dimensions sup\'erieures}, Proceedings of the {I}nternational {C}ongress of
  {M}athematicians, {V}ol. {II} ({B}eijing, 2002), Higher Ed. Press, Beijing,
  2002, pp.~405--414.

\bibitem{Go}
Robert Gompf, \emph{A new construction of symplectic manifolds}, Ann. Math.
  \textbf{142} (1995), no.~3, 527--595.

\bibitem{Gompf-handle}
\bysame, \emph{Handlebody construction of {S}tein surfaces}, Ann. Math.
  \textbf{148} (1998), no.~2, 619--693.

\bibitem{Gompf-Stipsicz}
Robert~E. Gompf and Andr\'{a}s~I. Stipsicz, \emph{{$4$}-manifolds and {K}irby
  calculus}, Graduate Studies in Mathematics, vol.~20, American Mathematical
  Society, Providence, RI, 1999.

\bibitem{HRRR}
Jonathan Hanselman, Jacob Rasmussen, Sarah~Dean Rasmussen, and Liam Watson,
  \emph{L-spaces, taut foliations, and graph manifolds}, Compos. Math.
  \textbf{156} (2020), no.~3, 604--612.

\bibitem{Ko}
Janos Koll\'ar, \emph{Toward moduli of singular varieties}, Comp. Math.
  \textbf{56} (1985), no.~1, 369--398.

\bibitem{LS3}
Paolo Lisca and Andr\'as~I. Stipsicz, \emph{Ozsv\'ath-{S}zab\'o invariants and
  tight contact 3-manifolds. {III}}, J. Symplectic Geom. \textbf{5} (2007),
  no.~4, 357--384.

\bibitem{Mat}
Irena Matkovi\v{c}, \emph{Classification of tight contact structures on small
  {S}eifert fibered {$L$}-spaces}, Algebr. Geom. Topol. \textbf{18} (2018),
  no.~1, 111--152. \MR{3748240}

\bibitem{Nem5}
Andr\'as N\'emethi, \emph{Five lectures on normal surface singularities},
  Bolyai Society Math. Studies \textbf{8} (1999), 269--351.

\bibitem{NemL}
\bysame, \emph{Links of rational singularities, {L}-spaces and {LO} fundamental
  groups}, Invent. Math. \textbf{210} (2017), no.~1, 69--83.

\bibitem{NT}
Andr\'as N\'emethi and Meral Tosun, \emph{Invariants of open books of links of
  surface singularities}, Stud. Sci. Math. Hung. \textbf{48} (2010), no.~1,
  135--144.

\bibitem{We2}
Klaus Niederkr\"uger and Chris Wendl, \emph{Weak symplectic fillings and
  holomorphic curves}, Ann. Sci. \'Ec. Norm. Sup\'er. \textbf{44} (2011),
  no.~5, 801--853.

\bibitem{Oba}
Takahiro Oba, \emph{A note on {Mazur} type {S}tein fillings of planar contact
  manifolds}, Topology Appl. \textbf{193} (2015), 302--308.

\bibitem{Oba2}
\bysame, \emph{Stein fillings of homology 3-spheres and mapping class groups},
  Geom. Dedicata \textbf{183} (2016), 69--80.

\bibitem{OhtaOno}
Hiroshi Ohta and Kaoru Ono, \emph{Simple singularities and symplectic
  fillings}, J. Differential Geom. \textbf{69} (2005), no.~1, 1--42.

\bibitem{OSS}
Peter Ozsv\'ath, Andr\'as Stipsicz, and Zolt\'an Szab\'o, \emph{Planar open
  books and {F}loer homology}, Int. Math. Res. Not. (2005), no.~54, 3385--3401.

\bibitem{OSlens}
Peter Ozsv\'ath and Zolt\'an Szab\'o, \emph{On knot {F}loer homology and lens
  space surgeries}, Topology \textbf{44} (2005), no.~6, 1281--1300.

\bibitem{RasmussenS}
Sarah~D. Rasmussen, \emph{L-space intervals for graph manifolds and cables},
  Comp. Math. \textbf{153} (2017), no.~5, 1008--1049.

\bibitem{Scho}
Stephan Sch\"onenberger, \emph{Determining symplectic fillings from planar open
  books}, J. Symplectic Geom. \textbf{5} (2007), no.~1, 19--41.

\bibitem{Tjurina}
G.~N. Tjurina, \emph{The topological properties of isolated singularities of
  complex spaces of codimension one}, Izv. Akad. Nauk SSSR Ser. Mat.
  \textbf{32} (1968), 605--620.

\bibitem{Tos}
B\"{u}lent Tosun, \emph{Tight small {S}eifert fibered manifolds with $e_0=
  -2$}, Algebr. Geom. Topol. \textbf{20} (2020), no.~1, 1--27.

\bibitem{Wand}
Andy Wand, \emph{Mapping class group relations, {S}tein fillings, and planar
  open book decompositions}, J. Topol. \textbf{5} (2012), no.~1, 1--14.

\bibitem{We}
Chris Wendl, \emph{Strongly fillable contact manifolds and {$J$}-holomorphic
  foliations}, Duke Math. J. \textbf{151} (2010), no.~3, 337--384.

\end{thebibliography}

\end{document}